\documentclass[12pt]{amsart}
\usepackage[utf8]{inputenc}

\usepackage{hyperref, a4wide}
\usepackage[british]{isodate}
\usepackage{amsfonts,amsmath,mathtools,amsthm,amssymb,mathrsfs}
\usepackage{graphicx, caption, subcaption}

\DeclarePairedDelimiter\abs{\lvert}{\rvert}
\DeclarePairedDelimiter\norm{\lVert}{\rVert}
\DeclareMathOperator{\diag}{diag}
\DeclareMathOperator{\ric}{Ric}
\DeclareMathOperator{\tr}{tr}

\theoremstyle{plain}
\newtheorem{thm}{Theorem}[section]

\newtheorem{lem}[thm]{Lemma}

\title{Computer-assisted construction of $SU(2)$-invariant negative Einstein metrics}
\author{Qiu Shi Wang}
\address{Mathematical Institute\\ University of Oxford \\
	Oxford\\
	OX2 6GG\\
	United Kingdom} \email{wangqs@maths.ox.ac.uk, qiu.s.wang@mail.mcgill.ca}
\cleanlookdateon
\date{29 April 2026}
\thanks{The author is supported by the Engineering and Physical Sciences Research Council [grant number EP/W524311/1] and the Fonds de recherche du Québec -- Nature et technologies [numéro de dossier 331732].}

\begin{document}
	\maketitle
		
\begin{abstract}
	We construct a 2-parameter family of new triaxial $SU(2)$-invariant complete negative Einstein metrics on the complex line bundle $\mathcal{O}(-4)$ over $\mathbb{C}P^1$. The metrics are conformally compact and neither Kähler nor self-dual. The proof involves using rigorous numerics to produce an approximate Einstein metric to high precision in a bounded region containing the singular orbit or ``bolt'', which is then perturbed to a genuine Einstein metric using fixed-point methods. At the boundary of this region, the latter metric is sufficiently close to hyperbolic space for us to show that it indeed extends to a complete, asymptotically hyperbolic Einstein metric.
\end{abstract}

\section{Introduction and preliminaries}

\subsection{Introduction}

An Einstein manifold is a smooth manifold $M$ equipped with a Riemannian metric $g$ satisfying
\begin{equation}\label{einstein}
	\ric g = \Lambda g
\end{equation}
for some real constant $\Lambda$. In local coordinates, (\ref{einstein}) is a second-order quasilinear PDE for which it is difficult to produce or classify solutions in general. Nonetheless, by requiring that $g$ possess certain symmetries, we can reduce the number of independent variables in (\ref{einstein}) and consequently substantially reduce its complexity. For instance, if $(M, g)$ is a homogeneous space, i.e. $g$ is invariant under the transitive action of a Lie group $G$, then the Einstein equation becomes a tractable algebraic condition \cite{WZ86}. There is a complete classification of homogeneous Einstein manifolds of dimension 4 \cite{J69}.

Less restrictively, one can assume that $G$ acts on $(M,g)$ with generic ``principal orbits'' of codimension one which together cover a dense subset of $M$. In this case, $g$ is said to be of cohomogeneity one and (\ref{einstein}) becomes a system of ODEs in a coordinate normal to the principal orbits. This system has been extensively studied and many new, often explicit, Einstein metrics have been found, for instance in \cite{BB82, DW00, EW00}. In particular, several papers, for instance \cite{B98, BH26, NW25}, have considered $(O(d_1+1)\times O(d_2+1))$-invariant doubly warped product cohomogeneity one metrics on the $(d_1+d_2+1)$-dimensional sphere. 

An intriguing feature of the work \cite{BH26} by Buttsworth--Hodgkinson is that it uses a somewhat unconventional, computer-assisted method. More precisely, the authors construct a non-round Einstein metric on $S^{12}$ in the following way. First, they obtain a heuristic solution up to high accuracy, which they approximate rigorously using Chebyshev polynomials in order to obtain a rigorous upper bound on the \textit{a posteriori} error in the Einstein equation. Then, they linearise the system around this Chebyshev-interpolated approximate solution and use fixed-point theory to show that there is a true solution ``nearby'' if the \textit{a posteriori} error is sufficiently small.

Unlike in the homogeneous case, we do not yet have a classification of cohomogeneity one Einstein 4-manifolds. In this work, we focus on the most interesting special case of the latter, namely $SU(2)$-invariant 4-dimensional cohomogeneity one metrics. They are less symmetric than $(O(d_1+1)\times O(d_2+1))$-invariant metrics, as they depend on 3 functions rather than 2, since the space of left-invariant metrics on $SU(2)$, equivalently the space of scalar products on the Lie algebra $\mathfrak{su}(2)$, is 3-dimensional after diagonalisation\footnote{In fact, the metric is constrained by the Einstein equation to remain diagonal \cite{D09}, so there are 3 functions rather than 6.}. The topology is also more complex than that of a doubly warped product. Einstein metrics of this form are relevant in physics, where they are known as ``Bianchi IX'' solutions \cite{GP79}, in reference to the Bianchi classification of real 3-dimensional Lie algebras. The $SU(2)$-invariant Einstein equation has been solved under a variety of additional symmetry assumptions, detailed in \S \ref{subsystemsubsection}. A summary of known examples of Bianchi IX metrics can be found in Table \ref{examplestable}.

The objective of the present work is to use the numerical analysis and computational methods of \cite{BH26} to construct new triaxial $SU(2)$-invariant complete negative Einstein metrics. The constructed metrics are asymptotically hyperbolic, of non-constant curvature, and the total space is the bundle $\mathcal{O}(-4)$ over $\mathbb{C}P^1$, although similar techniques can most likely be used to construct similar metrics on $\mathcal{O}(-1)$ or $\mathcal{O}(-2)$. Broadly speaking, our proof goes in two steps, both of which will involve some computer assistance : constructing a solution starting from the singular orbit (i.e. the zero section of the bundle $\mathcal{O}(-k)$ in question)(\S\ref{substep1}), and extending it to infinity (\S\ref{substep2}). Our numerical methods only allow us to analyse a particular, arbitrarily chosen metric in a conjectured 2-parameter family of Einstein metrics (see Figure \ref{b1hplot}). Nonetheless, by the continuous nature of our arguments, we in fact obtain a small 2-parameter region of (distinct even up to diffeomorphism, see \S\ref{extensionsubsection}) metrics in a small open neighbourhood of the abovementioned metric in moduli space.

\subsubsection{Solution starting from the singular orbit} \label{substep1}

The cohomogeneity one Einstein equation, considered as an initial value problem in the radial geodesic coordinate $t$, is well-posed and admits a solution in a neighbourhood of the singular orbit which is unique up to a finite number of parameters \cite{VZ24} (in the $SU(2)$-invariant case, 2 parameters). However, we would like to prove the existence of such a solution up to a quantified, positive time $t_f$, a result unavailable from standard theory. To do so, we first linearise the Einstein equation (suitably rewritten as a first-order singular initial value problem in the sense of \cite[Theorem 4.7]{FH17}) around a approximate solution which is ``nearly Einstein'' in the sense that the \textit{a posteriori} differential equation error is small. Practically, this approximate solution is obtained by fitting a high-order sum of Chebyshev polynomials to a high-precision heuristic numerical power series solution. From the approximate solution, we obtain using interval arithmetic rigorous bounds on the terms of the Einstein equation, allowing us to use the fixed point theorem \ref{fixedpointtheorem} to construct a smooth metric up to time $t_f$ with good control on its $C^0$ norm. This procedure broadly follows \cite{BH26}.

\subsubsection{Extending the metric to infinity} \label{substep2}

In order to perform numerics near infinity, we rescale the problem to a compact non-geodesic coordinate $r\in [0,1)$, with infinity at $r=1$. A key observation is that our new negative Einstein metrics are \textit{conformally compact} \cite{CG22}, so we work in a choice of ``conformally compactified'' variables. Then, we use a (computer-assisted) Grönwall's inequality argument to show that Einstein metrics which are ``close enough'' to hyperbolic space must tend towards it at infinity, and in particular must be complete.

\bigskip

Finally, we fix a particular choice of initial data and $t_f$, and produce a specific approximate solution which is ``close enough'' to bridge the gap between the arguments of \S\ref{substep1} and \S\ref{substep2}, using the specific values in appendix \ref{valuesapp}. This leads to the main result of this work, Theorem \ref{maintheorem}.

\textit{Remark.} In a sense, our work fulfils part of the hope of Buttsworth--Hodgkinson that their methods be applied to construct solutions to more general geometric equations. Moreover, given the highly procedural nature of the analytical approach and the great flexibility of computational tools and methods, we believe that the potential of this computer-assisted strategy is far from exhausted. In particular, it would be fruitful to apply the strategy to other cohomogeneity one special curvature conditions.

\subsection{Setup and notation}

Consider a smooth 4-manifold $M$ equipped with a Riemannian metric $g$ of cohomogeneity one under the action of $SU(2)$. Then, if $g$ is Einstein, it can be written as
\begin{equation}\label{b9}
	g=dt^2 + a(t)^2 \sigma_1^2 + b(t)^2\sigma_2^2 + c(t)^2\sigma_3^2,
\end{equation}
where $\sigma_1, \sigma_2,\sigma_3$ are $SU(2)$-invariant 1-forms on the principal orbits satisfying $d\sigma_1 = \sigma_2\wedge \sigma_3$, etc., and $t$ is a coordinate parametrising a unit speed geodesic normal to the principal orbits. The fact that (\ref{b9}) can be written in diagonal form follows from \cite{D09}. The group action can also be $SO(3)$ and the principal orbit can be $SU(2)$ or certain finite quotients thereof.

Let the dot denote the derivative $d/dt$. In the directions tangential to the principal orbits, the Einstein equation (\ref{einstein}) for the above metric (\ref{b9}) is
\begin{equation*}
	r_t - \dot L - (\tr L)L = \Lambda\; \mathrm{id},
\end{equation*}
where $r_t$ is the Ricci endomorphism of the principal orbits, and 
\begin{equation*}
	L=\diag \left( \frac{\dot a}{a}, \frac{\dot b}{b}, \frac{\dot c}{c}\right)
\end{equation*}
is their shape operator. In terms of $a,b,c$, the tangential Einstein equation is written as \cite{DS94} 
\begin{equation}\label{tangential}
	\frac{d}{dt}\left(\frac{\dot a}{a}\right) = -\frac{\dot a}{a}\left(\frac{\dot a}{a} + \frac{\dot b}{b}+ \frac{\dot c}{c}\right) + \frac{a^4 - (b^2-c^2)^2}{2a^2b^2c^2} - \Lambda,
\end{equation}
along with the equations obtained by cyclically exchanging $a,b,c$. In the radial ($\partial/\partial t$) direction, the Einstein equation is
\begin{equation*}
	-\tr (\dot L) + \tr(L^2) = \Lambda.
\end{equation*}
Taking the trace of the tangential equation and using the radial equation yields the conservation law
\begin{equation*}
	S + \tr (L^2) - (\tr L)^2 = 2\Lambda,
\end{equation*}
where $S$ is the scalar curvature of the principal orbits. In terms of $a,b,c$, it takes the form
\begin{equation}\label{conservation}
	\frac{-a^4 - b^4 - c^4 + 2a^2b^2 + 2b^2c^2 + 2c^2a^2}{2(abc)^2} = 2\left( \frac{\dot a \dot b}{ab} + \frac{\dot b \dot c}{bc} + \frac{\dot c \dot a}{ca}\right) + 2\Lambda.
\end{equation}
If the tangential equations are satisfied, then it suffices to check that the radial equation, or equivalently the conservation law, holds at one time to show that it holds at all other times \cite{EW00}. It is also worth noting that the Einstein ODEs are invariant under any sign change $(a,b,c)\rightarrow(\pm a, \pm b, \pm c)$, and thus we may assume that $a,b,c\geq 0$.

The ODEs (\ref{tangential}) form a singular initial value problem at the singular orbit(s) of the manifold. Suppose throughout the remainder of the paper that $M$ is complete. 

\begin{itemize}
	\item For $\Lambda>0$, by the Bonnet--Myers theorem, $M$ is compact and consequently has two singular orbits.
	\item When $\Lambda = 0$, there is exactly one singular orbit, for the following reasons. There at least one, because otherwise, by the Cheeger--Gromoll splitting theorem, $M$ is isometric to a product $N\times \mathbb{R}$. There is at most one, as otherwise $M$ is compact, all Killing vector fields on $M$ are parallel, which leads to a contradiction if one considers the Killing field generated by the cohomogeneity one action. 
	\item When $\Lambda<0$, As compact negative Einstein manifolds admit no nonzero Killing vector fields, there can be either 0 or 1 singular orbits. While complete metrics with no singular orbits, i.e. solutions existing for all $t\in \mathbb{R}$, may exist, we will restrict ourselves in the present work to metrics with at least one singular orbit.
\end{itemize}

In addition, the topology of the singular and principal orbits imposes boundary conditions required for smooth extension of the metric to the singular orbit. More specifically, there are four possible boundary conditions at the singular orbit for $g$ to be a smooth metric. In each case, the space of smooth Einstein metrics defined on a neighbourhood of the singular orbit, considered as formal power series for $a$, $b$ and $c$, is two-dimensional and can be parametrised by some (non-canonical choice of a) pair of boundary conditions. Here, the nuts and bolts terminology follows \cite{GP79}, and we will assume that the singular orbit in question is at $t=0$.

\begin{itemize}
	\item \textbf{``Nut'' case : } We have a one-point singular orbit, principal orbits homeomorphic to $S^3$ and the initial data $a(0) = b(0) = c(0) = 0$ and $\dot a(0) = \dot b(0) = \dot c(0) = \frac{1}{2}$. The global topology of the manifold is then $\mathbb{R}^4$. The space of local Einstein metrics, i.e. Einstein metrics defined in a neighbourhood of $t=0$, can be parametrised by $a^{(3)}(0)$ and $b^{(3)}(0)$.
	\item \textbf{``Bolt'' cases : } For $k=1, 2, 4$, up to double covers, we have an $S^2$ singular orbit and principal orbits of the form $S^3/\mathbb{Z}_k$. The total space is consequently the complex line bundle $\mathcal{O}(-k)$ over $\mathbb{C}P^1$. For the initial data, we may assume without loss of generality that $a(0)=0$ but $b(0), c(0)\neq 0$. Then, it follows that $b(0)=c(0)=h>0$ and we have one of the following cases.
	\begin{itemize}
		\item $\mathcal{O}(-4)$ : $\dot a(0) = 2$. The space of local Einstein metrics is parametrised by $h$ and $\dot b(0)$.
		\item $\mathcal{O}(-2)$ : $\dot a(0) = 1$. The parameters are $h$ and $\ddot b(0)$.
		\item $\mathcal{O}(-1)$ : $\dot a(0) = \frac{1}{2}$. The parameters are $h$ and $b^{(4)}(0)$.
	\end{itemize}
	The parameter $h$ can be interpreted as the size of the $S^2$ bolt.
\end{itemize}

Details and derivations for the above boundary conditions and free parameters can be found in appendix \ref{boundaryappendix}. Examples of (noncompact) Einstein metrics for each topology are given in Table \ref{examplestable}.

\subsection{Known examples and results for subsystems}\label{subsystemsubsection}

The $SU(2)$-invariant Einstein equations (\ref{tangential}), (\ref{conservation}) have been fully solved under a variety of different additional symmetry or special curvature conditions. This includes
\begin{itemize}
	\item The Kähler condition \cite{DS94}.
	\item The (anti-)self-dual condition \cite{AH88, H95}.
	\item An additional $U(1)$ symmetry, i.e. $g$ is $U(2)$-invariant. The equations under this hypothesis, given by (\ref{tangential}), (\ref{conservation}) with $b=c$, are fully integrable \cite{BB82}. Their solutions are called \textit{biaxial metrics}. If $a,b,c$ are distinct, the metric is said to be \textit{triaxial}.
\end{itemize}
\begin{table}[h!]
	\begin{tabular}{|c|c|c|c|c|c|}
		\hline
		\multicolumn{2}{|c|}{positive} & \multicolumn{2}{c|}{Ricci-flat} & \multicolumn{2}{c|}{negative}  \\ \hline
		metric & topology & metric & topology & metric & topology \\ \hline
		round $g_{S^4}$ & $S^4$ & Taub--NUT \cite{NTU63} & $\mathbb{R}^4$ & Hitchin* \cite{H95} & $\mathbb{R}^4$\\
		$g_{S^2}\times g_{S^2}$ & $S^2\times S^2$ & Atiyah--Hitchin* \cite{AH88} & $\mathcal{O}(-4)$ & Pedersen \cite{P85} & $\mathcal{O}(-4)$ \\
		Fubini--Study & $\mathbb{C}P^2$ & Eguchi--Hanson \cite{EH79} & $\mathcal{O}(-2)$ & {\footnotesize Dancer--Strachan* \cite{DS94}} & $\mathcal{O}(-2)$ \\
		Page \cite{P78} & $\mathbb{C}P^2\# \overline{\mathbb{C}P^2}$ & Taub-bolt \cite{P78a} & $\mathcal{O}(-1)$ & & \\ \hline
	\end{tabular}
	\caption{List of known $SU(2)$-invariant Einstein metrics. Triaxial metrics are indicated by a star *. The author conjectures, based on numerical evidence, that no additional positive or Ricci-flat examples exist (consistently with \cite[Conjecture 1.5]{D25} in the positive case).}
	\label{examplestable}
\end{table}

\subsection{Summary of numerical results for negative Einstein metrics}
We can obtain heuristic numerics for the Einstein equation (\ref{tangential}) starting from the singular orbit by starting a Runge--Kutta solver at a very small time, using the first few terms of the formal power series solutions in appendix \ref{boundaryappendix} as approximate initial data. This is repeated along a grid of points in the parameter plane of local Einstein metrics. Numerical results for the $\dot a(0) = 2$ case, which will be the focus of the remainder of the paper, are shown in Figure \ref{b1hplot}.

We find that for each of the four singular orbit topologies there appears to be a 2-parameter family of complete negative Einstein metrics. Generically, these metrics are all conformally compact, with conformal infinity given by (the conformal class of) an $SU(2)$-invariant metric on $S^3$ or some finite quotient thereof. They are also generically triaxial, with a 1-dimensional biaxial subfamily. In some sense, the metrics with bolt singular orbits are generalisations of the Graham--Lee fillings \cite{GL91} but with nontrivial internal topology. We numerically assess whether the metrics are Kähler using criteria in \cite{DS94}, and whether they are self-dual using criteria in \cite{CGLP03}. 

More specifically, we observe numerically that there are negative Einstein metrics as follows.

\begin{itemize}
	\item For a point singular orbit (nut, $M=\mathbb{R}^4$), there is a 2-parameter family of anti-self-dual metrics which appears in Hitchin's work \cite{H95}. This includes hyperbolic space and the pseudo-Fubini--Study metric \cite{P85}.
	\item When the singular orbit is $S^2$ and the principal orbit is $SO(3)/\mathbb{Z}_2$ (bolt, $M=\mathcal{O}(-4)$), we see a 2-parameter family of complete solutions, out of which there is a (biaxial) Kähler metric which appears in a work of H. Pedersen \cite{P85}. The rest of the triaxial metrics (see Figure \ref{b1hplot}), all non-Kähler and non self-dual, appear to be new.
	\item When the singular orbit is $S^2$ and the principal orbit is $SO(3)$ (bolt, $M=\mathcal{O}(-2)$), we see a 2-parameter family of non-self-dual complete metrics, out of which there is a 1-parameter family of Kähler metrics due to Dancer--Strachan \cite{DS94}.
	\item When the singular orbit is $S^2$ and the principal orbit is $SU(2)$ (bolt, $M=\mathcal{O}(-1)$), we see a 2-parameter family of complete metrics which are not self-dual nor Kähler.
\end{itemize}

\textit{Remarks.} Numerical evidence for the existence of these two-parameter families of metrics has been independently found in the recent work \cite{D25}. The author of the latter work also proves their existence in an open neighbourhood of the 1-dimensional subset of parameter space corresponding to $U(2)$-invariant metrics using perturbative methods. However, the procedure in the present work allows the construction of metrics that are genuinely triaxial, in the sense of being ``arbitrarily far away'' from being biaxial/$U(2)$-invariant. 

Our proof method, however, is limited to constructing metrics in an infinitesimal neighbourhood of any particular point in the shaded region of Figure \ref{b1hplot}. Consequentially, we do not obtain any information as to what happens at the boundaries of the region, nor where that boundary lies. These questions warrant further investigation using different methods.

\subsection{Plan of the paper}

\begin{figure}
	\includegraphics[width = .65\textwidth]{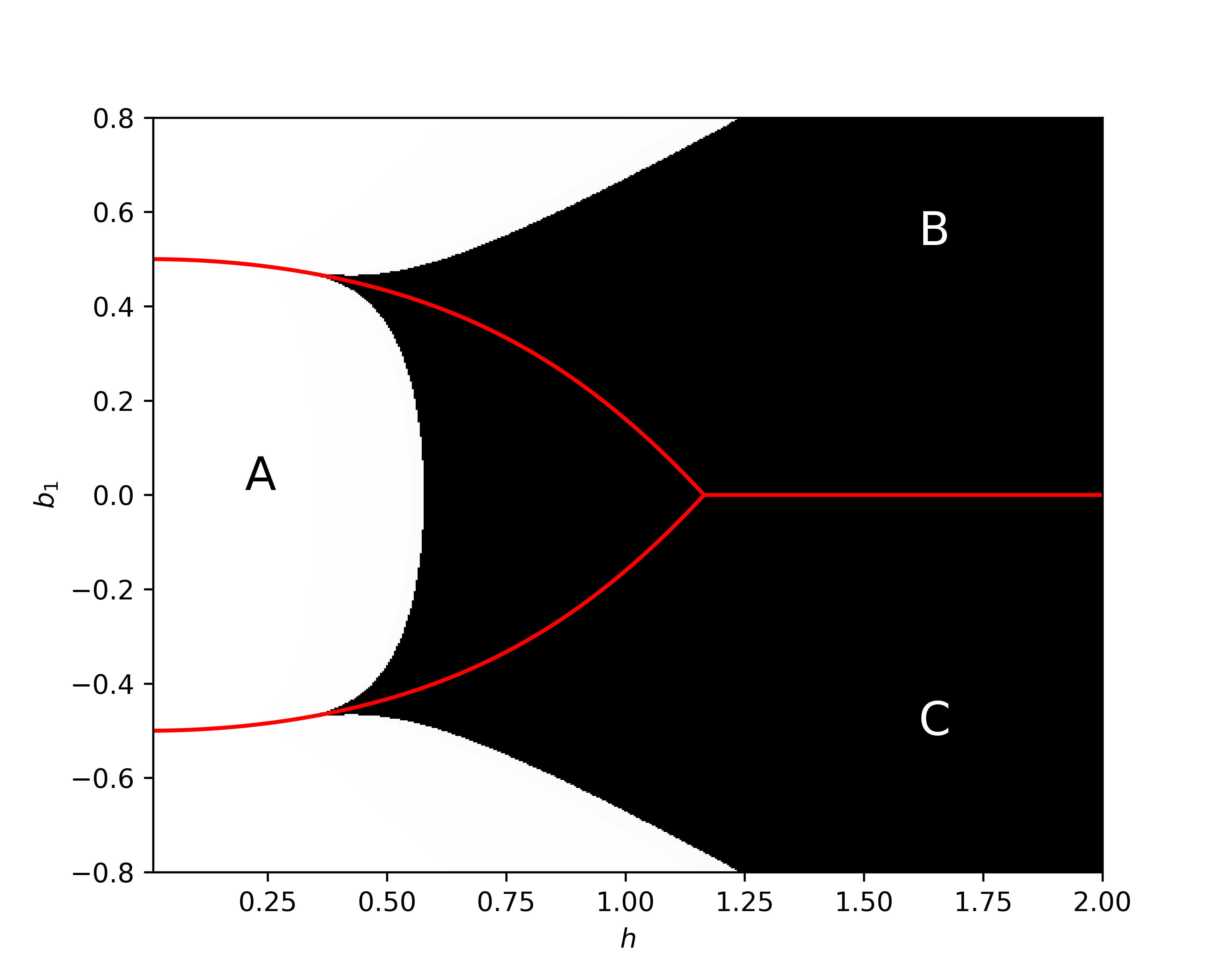}
	\caption{Numerical evidence for the existence of complete negative Einstein metrics in the parameter plane $(h,b_1)$ for an $\dot a(0)=2$ bolt, corresponding to the total space $M=\mathcal{O}(-4)$. The dark shaded region corresponds to likely complete Einstein metrics. The red curves divide the half-plane into three regions $A$, $B$ and $C$, in which $a, b$ and $c$ are respectively the largest warping functions as one approaches either conformal infinity (shaded region) or a finite time singularity (unshaded region). Note that there is a symmetry $b_1\leftrightarrow -b_1$, $b\leftrightarrow c$, and the horizontal axis $b_1=0$ corresponds to biaxial metrics with $b\equiv c$.}
	\label{b1hplot}
\end{figure}

We will use a two-pronged computer-assisted scheme to construct negative Einstein metrics whose existence is numerically evidenced by Figure \ref{b1hplot}. First, noting that these metrics appear to be asymptotically hyperbolic, we will study them near infinity using the notion of conformally compact metrics, introduced in \S\ref{confcompactsection}. Next, in \S\ref{infinitysection}, we will use a Grönwall's inequality-type estimate to show in Lemma \ref{infinitylemma} that ``sufficiently nice'' negative Einstein metrics near infinity can be extended to complete metrics.

At this stage, we specialise to the $k=4$, or equivalently $\dot a(0)=2$ case, and detail the singular initial value problem corresponding to the Einstein equation in \S\ref{ivpsection}. We then use in \S\ref{fixedpointsection} a fixed-point argument along the lines of \cite{BH26} to prove the existence of a true Einstein metric satisfying hypotheses (\ref{inf2})-(\ref{inf4}) of Lemma \ref{infinitylemma} up until some \textit{a priori} determined time $t_f$, given a ``sufficiently nearby'' approximate solution in the sense of the hypotheses of Lemma \ref{existencelemma}.

Finally, in \S\ref{computersection}, we describe the computer-assisted construction of the approximate solution, and how it leads to the main result of the paper, Theorem \ref{maintheorem}.

We consider only the $\mathcal{O}(-4)$ case for simplicity and for illustrative purposes, but a similar procedure can most likely be used to construct new triaxial metrics on $\mathcal{O}(-2)$ or $\mathcal{O}(-1)$. Throughout the construction, the Einstein constant is fixed to be $\Lambda = -3$.

\section{Negative Einstein metrics in the conformally compact formalism}\label{confcompactsection}

The negative Einstein metrics we would like to construct are noncompact, so there is \textit{a priori} no straightforward way of proving completeness using the ODEs in the unit-speed radial geodesic coordinate $t$, which goes to infinity. However, an important observation is that the metrics are conformally compact, which allows us to formulate an initial value problem at infinity. In this section, we recall the relevant definitions and reformulate the equations (\ref{tangential}), (\ref{conservation}).

Let $M$ be a manifold with boundary. A metric $g$ on the interior of $M$ is said to be \textit{conformally compact} if there is a smooth \textit{defining function} $\rho$ such that $\bar g = \rho^2 g$ extends smoothly to the boundary $\partial M$. A defining function $\rho$ satisfies $\rho>0$ in the interior of $M$, $\rho=0$ on $\partial M$, and $D\rho \neq 0$ on $\partial M$ (i.e. its gradient is nonvanishing on the boundary). The conformal class of the restriction of $\bar g$ to $\partial M$ is said to be the \textit{conformal infinity} of the conformally compact manifold $(M,g)$.

The metric $g$ is said to be \textit{conformally compact Einstein} if in addition $\ric g = -(n-1)g$, where $n=\dim M$. There is a considerable amount of work on conformally compact Einstein manifolds in the literature; see for instance the recent review \cite{CG22}.

In the ($\Lambda = -3$) Bianchi IX case, we make the radial change of coordinates $r=\tanh(t/2)$ and consider the defining function $\rho = \frac{1}{2}(1-r^2)$. Note that by doing so, we compactify the manifold by adding a copy of a principal orbit $SU(2)/\mathbb{Z}_k$ at $r=1$, which corresponds to $t=\infty$. Indeed, let 
\begin{equation*}
	\alpha = \rho a, \qquad \beta = \rho b, \qquad \gamma = \rho c
\end{equation*}
be the defining functions of $\bar g = \rho^2 g = dr^2 + \alpha^2 \sigma_1^2 + \beta^2 \sigma_2^2 + \gamma^2 \sigma_3^2$. Let $'$ denote $d/dr$. Then, the tangential ODEs become
\begin{equation}\label{tangentialconformal}
	\left(\frac{\alpha'}{\alpha}\right)' = -\frac{\alpha'}{\alpha}\left(\frac{\alpha'}{\alpha} + \frac{\beta'}{\beta} + \frac{\gamma'}{\gamma}\right) + \frac{\alpha^4 - (\beta^2-\gamma^2)^2}{2\alpha^2\beta^2\gamma^2} + \frac{1}{\rho}\left( 5 - r\left(\frac{3\alpha'}{\alpha}+ \frac{\beta'}{\beta} + \frac{\gamma'}{\gamma}\right)\right)
\end{equation}
along with the cyclic permutations of $\alpha, \beta, \gamma$. Using the fact that 
\begin{equation}\label{confcompacttransform}
	\frac{\dot a}{a} = \rho\left(\frac{\alpha'}{\alpha} + \frac{r}{\rho}\right),
\end{equation}
we obtain that the conservation law (\ref{conservation}) takes the form
\begin{multline}\label{conservationr}
	\frac{-\alpha^4 - \beta^4 - \gamma^4 + 2\alpha^2\beta^2 + 2\beta^2\gamma^2 + 2\gamma^2\alpha^2}{2(\alpha\beta\gamma)^2} =\\
	2\left(\left(\frac{\alpha'}{\alpha} + \frac{r}{\rho}\right)\left(\frac{\beta'}{\beta} + \frac{r}{\rho}\right) + \left(\frac{\beta'}{\beta} + \frac{r}{\rho}\right)\left(\frac{\gamma'}{\gamma} + \frac{r}{\rho}\right) +\left(\frac{\gamma'}{\gamma} + \frac{r}{\rho}\right) \left(\frac{\alpha'}{\alpha} + \frac{r}{\rho}\right)\right) - \frac{6}{\rho^2}.
\end{multline}
Complete, conformally compact Einstein metrics correspond to solutions of (\ref{tangentialconformal}) on $r\in [0,1]$ satisfying the above conservation law.

\textit{Remark.} A complete Einstein metric as above has conformal infinity given by the class of a homogeneous metric on $\partial M = SU(2) / \mathbb{Z}_k$, $k=1, 2, 4$, parametrised by $(\alpha(1), \beta(1), \gamma(1))$.

\section{Proof of existence : asymptotic hyperbolicity near infinity}\label{infinitysection}

The solutions to (\ref{tangential}) for $\Lambda<0$ are ``stable at hyperbolic infinity'' in the sense that if an Einstein metric is sufficiently close to the hyperbolic metric, then it extends to infinity in a way that approaches hyperbolic space. In this section, we will make this notion precise in Lemma \ref{infinitylemma} using a Grönwall's inequality argument.

In the coordinate $s\equiv 1-r$, the tangential equations (\ref{tangentialconformal}) become, with $'$ now denoting $d/ds$ instead of $d/dr$, 
\begin{equation*}
	\left(\frac{\alpha'}{\alpha}\right)' = -\frac{\alpha'}{\alpha}\left(\frac{\alpha'}{\alpha} + \frac{\beta'}{\beta} + \frac{\gamma'}{\gamma}\right) + \frac{\alpha^4 - (\beta^2 - \gamma^2)^2}{2\alpha^2\beta^2\gamma^2} + \frac{1}{\rho}\left(5 + (1-s)\left(\frac{3\alpha'}{\alpha} + \frac{\beta'}{\beta} + \frac{\gamma'}{\gamma}\right)\right)
\end{equation*}
and cyclically. The form of the equations requires that solutions which extend to $s=0$ satisfy
\begin{equation*}
	\frac{\alpha'(0)}{\alpha(0)} = \frac{\beta'(0)}{\beta(0)} = \frac{\gamma'(0)}{\gamma(0)} = -1.
\end{equation*}
Consequently, we make the change of variables 
\begin{gather*}
	X_1=\frac{1}{\alpha}, \quad X_2=\frac{1}{\beta}, \quad X_3=\frac{1}{\gamma},\\
	Z_1=\frac{\alpha'}{\alpha}+1, \quad Z_2=\frac{\beta'}{\beta}+1, \quad Z_3=\frac{\gamma'}{\gamma}+1.
\end{gather*}
Then, the system becomes
\begin{equation}\label{infinitysystem}
	\begin{split}
	X_1' &= (1-Z_1)X_1\\
	Z_1' &= \frac{1}{s}(3Z_1 + Z_2 + Z_3) - Z_1(Z_1 + Z_2 + Z_3 - 3) \\
	&\quad + \frac{5s}{2-s} + \frac{1-s}{2-s}(3Z_1 + Z_2 + Z_3) + R_1
	\end{split}
\end{equation}
and cyclic permutations, where
\begin{equation*}
	R_1 = \frac{\alpha^4 - (\beta^2 - \gamma^2)^2}{2\alpha^2\beta^2\gamma^2}=\frac{1}{2}\left(\frac{\alpha^2}{\beta^2\gamma^2} - \frac{\beta^2}{\gamma^2\alpha^2} - \frac{\gamma^2}{\alpha^2\beta^2}\right) + \frac{1}{\alpha^2}
\end{equation*}
and cyclically. Denote $Z=(Z_1,Z_2,Z_3)^T$ and 
\begin{equation*}
	M_{-1}=\begin{pmatrix*} 3 & 1 & 1\\
	1 & 3 & 1\\
	1 & 1 & 3\end{pmatrix*}.
\end{equation*}
We multiply the equation for $Z_i'$ by $Z_i$ and sum over $i=1,2,3$ to obtain
\begin{multline*}
	\frac{1}{2}(\abs{Z}^2)' = \frac{1}{s} Z^T M_{-1} Z + \abs{Z}^2\left(4 - \frac{s}{2-s} - (Z_1 + Z_2 + Z_3)\right) + \frac{5s}{2-s}(Z_1+Z_2+Z_3) \\
	+ \frac{1-s}{2-s}(Z_1+Z_2+Z_3)^2 + Z_1R_1 + Z_2R_2 + Z_3R_3.
\end{multline*}

Suppose that $s\leq s_0$ for some $s_0\in(0,1)$. We would like to estimate $(\abs{Z}^2)'$ from below in order to use Grönwall's inequality. 

From the Peter-Paul ($2ab\leq \varepsilon^{-1}a^2 + \varepsilon b^2$ for each $a,b\in \mathbb{R}$ and $\varepsilon>0$) and Cauchy--Schwarz inequalities we obtain the estimates
\begin{align}
	Z_1 + Z_2 + Z_3 &\geq -\frac{1}{2A^2}\abs{Z}^2 - \frac{3}{2}A^2\label{Aineq}\\
	Z_1 R_1 + Z_2 R_2 + Z_3 R_3 &\geq -\frac{1}{2B^2} \abs{Z}^2 - \frac{B^2}{2}(R_1^2+R_2^2+R_3^2),\label{Bineq}
\end{align}
which hold uniformly in $s$ for any constants $A, B>0$, to be chosen later. We discard the singular term by using the fact that $M_{-1}$ is positive definite, thus obtaining the estimate
\begin{multline*}
	\frac{1}{2}(\abs{Z}^2)'\geq \abs{Z}^2\left(4 - \frac{s}{2-s} - (Z_1 + Z_2 + Z_3)\right) \\
	+ \frac{5s}{2-s}\left(-\frac{1}{2A^2}\abs{Z}^2 - \frac{3}{2}A^2\right) -\frac{1}{2B^2} \abs{Z}^2 - \frac{B^2}{2}(R_1^2+R_2^2+R_3^2).
\end{multline*}

We now consider the following assumptions, for some constants $C\in\mathbb{R}$, $D\geq 0$:
\begin{itemize}
	\item \textit{Assumption C: } $Z_1+Z_2+Z_3 - (4-\frac{s}{2-s})\leq C$.
	\item \textit{Assumption D: } $R_1^2 + R_2^2 + R_3^2 \leq D$.
\end{itemize}
Suppose that both assumptions hold at $s=s_0$. Define the set
\begin{equation*}
	S_{C,D} = \left\{s\in[0,s_0]\; \lvert\; \text{assumptions C and D hold}\right\}.
\end{equation*}
Let $s_1 = \sup\{s\in [0,s_0]\;\lvert\; s\notin S_{C,D}\}$. Then, $[s_1,s_0]\subset S_{C,D}$, since $S_{C,D}$ is closed by the continuity of $Z_i$ and $R_i$. Our goal is to show that for suitable values of $Z(s_0)$ and of the constants $A, B, C, D$, such an $s_1\geq 0$ cannot exist and consequently $S_{C,D} = [0,s_0]$. If this holds, then by assumption C, $Z$ would be bounded up until $s=0$ and therefore $\alpha, \beta, \gamma$ would be bounded and the metric would exist up to conformal infinity. To prove this, we estimate from above the growth of $Z$ and $R$ in $S_{C,D}$, and show that they cannot ``break free'' from the assumptions within $[s_1, s_0]\subset [0, s_0]$.

\subsection{Estimates under assumptions C and D}\label{CDsubsection}
All the estimates in this subsection use assumptions C and D, and consequently hold in $s\in [s_1, s_0]$. Within the latter interval, we have the estimate
\begin{equation*}
	\frac{1}{2}(\abs{Z}^2)'\geq -C\abs{Z}^2 + \frac{5s}{2-s}\left(-\frac{1}{2A^2}\abs{Z}^2 - \frac{3}{2}A^2\right) -\frac{1}{2B^2} \abs{Z}^2 - \frac{B^2}{2}D.
\end{equation*}

For convenience in applying Grönwall's inequality, we change sign again to $\tau\equiv -s = r-1$. Then, the estimate becomes
\begin{equation*}
	\frac{d}{d\tau}\abs{Z}^2 \leq \abs{Z}^2\left(2C + \frac{1}{B^2} - \frac{5\tau}{2+\tau}\frac{1}{A^2}\right) + B^2D - \frac{15A^2\tau}{2+\tau}.
\end{equation*}
We may write this schematically as
\begin{equation}\label{gronwallschematic}
	\frac{d}{d\tau}\abs{Z}^2\leq \theta(\tau)\abs{Z}^2+\phi(\tau)
\end{equation}
with
\begin{equation*}
	\theta(\tau) = 2C + \frac{1}{B^2} - \frac{5\tau}{2+\tau}\frac{1}{A^2}, \qquad \phi(\tau) = B^2D - \frac{15A^2\tau}{2+\tau}.
\end{equation*}
Note that $\theta, \phi$ are defined up to $\tau=0$. We recall Grönwall's inequality:

\begin{thm}\label{gronwall}
	Let $\alpha, \beta, u$ be continuous functions on $[a,b]$, with $\beta\geq 0$. If, for all $t\in[a,b]$, $u$ satisfies
	\begin{equation*}
		u(t)\leq \alpha(t) + \int_a^t \beta(z) u(z)\, dz,
	\end{equation*}
	then
	\begin{equation}\label{gronwallineq}
		u(t)\leq \alpha(t) + \int_a^t \alpha(z)\beta(z)\exp\left(\int_z^t\beta(r)\, dr\right)dz.
	\end{equation}
\end{thm}

We integrate (\ref{gronwallschematic}) to obtain, denoting $Z_0 = Z(\tau = -s_0)$, 
\begin{equation*}
	\abs{Z}^2(\tau)-\abs{Z_0}^2 \leq \int_{-s_0}^{\tau} \phi(t) dt + \int_{-s_0}^\tau \theta(t) \abs{Z}^2(t)\, dt.
\end{equation*}
We can now apply Theorem \ref{gronwall} with $\tau$ being the time coordinate, $a=-s_0$,
\begin{equation*}
	u = \abs{Z}^2(\tau), \qquad \alpha=\abs{Z_0}^2+\int_{-s_0}^{\tau} \phi(t)\, dt, \qquad\beta = \theta(\tau). 
\end{equation*}
Explicitly, we get
\begin{equation}\label{zetadfn}
	\abs{Z}^2(\tau) \leq \abs{Z_0}^2 + \int_{-s_0}^\tau \phi(z)\, dz + \int_{-s_0}^\tau \left(\abs{Z_0}^2 + \int_{-s_0}^z\phi(t)\, dt\right)\theta(z) \exp\left(\int^\tau_z \theta(r)\, dr\right) dz.
\end{equation}

Denote the right-hand side of (\ref{zetadfn}) by $\zeta^2$. Suppose that $\theta>0$. Then, both integrands of the right-hand side are positive functions of $z$, and consequently $\zeta$ is an increasing function of $\tau$. From this point on, we return to using the coordinate $s=-\tau$. We consequently have
\begin{equation*}
	\abs{Z(s)}\leq \zeta(s)
\end{equation*}
for $s\in[s_1, s_0]$. In fact, we note that since $s_1\geq 0$, we have for any $s$ that
\begin{equation*}
	\zeta(s)\leq \zeta_m,
\end{equation*}
where
\begin{equation*}
	\zeta_m \equiv \abs{Z_0}^2 + \int_{-s_0}^0 \phi(z)\, dz + \int_{-s_0}^0 \left(\abs{Z_0}^2 + \int_{-s_0}^z\phi(t)\, dt\right)\theta(z) \exp\left(\int^0_z \theta(r)\, dr\right) dz
\end{equation*}
is a constant that depends only on $s_0, Z(s_0), A, B, C, D$. In turn, for each $s$,
\begin{equation*}
	\left\vert\frac{\alpha'}{\alpha} + 1\right\vert\leq \zeta_m,
\end{equation*}
which gives after integrating that
\begin{equation*}
	\alpha(s_0) e^{s_0-s}e^{-(s_0-s)\zeta_m} \leq \alpha(s) \leq \alpha(s_0) e^{s_0-s}e^{(s_0-s)\zeta_m}
\end{equation*}
for each $s\in [s_1, s_0]$. The same estimate holds for $\beta$ and $\gamma$.

We now turn to estimates of $R_i$ by expanding
\begin{multline*}
	R_1^2+R_2^2+R_3^2=\frac{1}{2\alpha^4} + \frac{1}{2\beta^4} + \frac{1}{2\gamma^4} + \frac{1}{\alpha^2\beta^2} + \frac{1}{\gamma^2\alpha^2} + \frac{1}{\beta^2\gamma^2}
	+ \frac{3\alpha^4}{4\beta^4\gamma^4} + \frac{3\beta^4}{4\gamma^4\alpha^4} + \frac{3\gamma^4}{4\alpha^4\beta^4} \\
	- \frac{\beta^2}{\gamma^2\alpha^4} - \frac{\gamma^2}{\beta^2\alpha^4} - \frac{\alpha^2}{\gamma^2\beta^4} - \frac{\gamma^2}{\alpha^2\beta^4} - \frac{\alpha^2}{\beta^2\gamma^4} - \frac{\beta^2}{\alpha^2\gamma^4}.
\end{multline*}
Using the Cauchy--Schwarz inequality, we estimate
\begin{align*}
	R_1^2+R_2^2+R_3^2&\leq -\frac{3}{2\alpha^4} - \frac{3}{2\beta^4} - \frac{3}{2\gamma^4} + \frac{1}{\alpha^2\beta^2} + \frac{1}{\gamma^2\alpha^2} + \frac{1}{\beta^2\gamma^2}
	+ \frac{3\alpha^4}{4\beta^4\gamma^4} + \frac{3\beta^4}{4\gamma^4\alpha^4} + \frac{3\gamma^4}{4\alpha^4\beta^4}\\
	&\leq -\frac{1}{2\alpha^4} - \frac{1}{2\beta^4} - \frac{1}{2\gamma^4} + \frac{3\alpha^4}{4\beta^4\gamma^4} + \frac{3\beta^4}{4\gamma^4\alpha^4} + \frac{3\gamma^4}{4\alpha^4\beta^4}\\
	&= \frac{3}{8\gamma^4}\left(\frac{\alpha^4}{\beta^4} + \frac{\beta^4}{\alpha^4} - \frac{4}{3}\right) + \frac{3}{8\beta^4}\left(\frac{\alpha^4}{\gamma^4} + \frac{\gamma^4}{\alpha^4} - \frac{4}{3}\right) + \frac{3}{8\alpha^4}\left(\frac{\beta^4}{\gamma^4} + \frac{\gamma^4}{\beta^4} - \frac{4}{3}\right)\\
	&\equiv K_R.
\end{align*}
The quantity $K_R$ is suitable for our estimates, as the expressions within parentheses are all strictly positive. Let $\alpha_0=\alpha(s_0)$, $\beta_0 = \beta(s_0)$, $\gamma_0 = \gamma(s_0)$, and suppose that $\alpha_0\geq \beta_0\geq \gamma_0$. Since $x+1/x$ is greatest at the maximum of either $x$ or $1/x$, whichever is larger, we conclude that
\begin{multline*}
	K_R (s_1) \leq \frac{3}{8}e^{(4\zeta_m-4)(s_0-s_1)}\bigg(\left(\frac{\alpha_0^4}{\beta_0^4}e^{8(s_0-s_1)\zeta_m}+ \frac{\beta_0^4}{\alpha_0^4}e^{-8(s_0-s_1)\zeta_m}- \frac{4}{3}\right) \\
	+ \left(\frac{\alpha_0^4}{\gamma_0^4}e^{8(s_0-s_1)\zeta_m}+ \frac{\gamma_0^4}{\alpha_0^4}e^{-8(s_0-s_1)\zeta_m}- \frac{4}{3}\right) + \left(\frac{\beta_0^4}{\gamma_0^4}e^{8(s_0-s_1)\zeta_m}+ \frac{\gamma_0^4}{\beta_0^4}e^{-8(s_0-s_1)\zeta_m}- \frac{4}{3}\right)\bigg).
\end{multline*}

For all $s_1$, this expression can be loosely estimated by
\begin{equation*}
	K_R\leq \begin{cases*}
		\frac{3}{8\gamma_0^4}\left(\frac{\alpha_0^4}{\beta_0^4} + \frac{\beta_0^4}{\alpha_0^4} \right) + \frac{3}{8\beta_0^4}\left(\frac{\alpha_0^4}{\gamma_0^4} + \frac{\gamma_0^4}{\alpha_0^4}\right) + \frac{3}{8\alpha_0^4}\left(\frac{\beta_0^4}{\gamma_0^4} + \frac{\gamma_0^4}{\beta_0^4} \right) & if $\zeta_m< 1/3$\bigskip\\
		\frac{3}{8}e^{(4\zeta_m-4)s_0}\bigg(\left(\frac{\alpha_0^4}{\beta_0^4}e^{8s_0\zeta_m}+ \frac{\beta_0^4}{\alpha_0^4}e^{-8s_0\zeta_m}\right) \\
		+ \left(\frac{\alpha_0^4}{\gamma_0^4}e^{8s_0\zeta_m}+ \frac{\gamma_0^4}{\alpha_0^4}e^{-8s_0\zeta_m}\right) + \left(\frac{\beta_0^4}{\gamma_0^4}e^{8s_0\zeta_m}+ \frac{\gamma_0^4}{\beta_0^4}e^{-8s_0\zeta_m}\right)\bigg) & if $\zeta_m \geq 1/3$.
	\end{cases*}
\end{equation*}
Denote the top value by $K_0$ and the lower value by $K_1$. We have thus laid the groundwork for the following lemma.

\begin{lem}\label{infinitylemma}
	Let $Z$ be a solution to the Einstein equation (\ref{infinitysystem}) in a neighbourhood of $s_0\in (0,1)$ and let $A,B>0$ be constants. Let $C\in \mathbb{R}$, $D>0$ be such that assumptions $C$ and $D$ hold at $s=s_0$, recalling that
	\begin{itemize}
		\item \textit{Assumption C: } $Z_1+Z_2+Z_3 - (4-\frac{s}{2-s})\leq C$.
		\item \textit{Assumption D: } $R_1^2 + R_2^2 + R_3^2 \leq D$.
	\end{itemize}
	Suppose that the following hypotheses, which depend only on $Z(s_0)$ and the choices of $A, B, C, D$, hold:
	\begin{enumerate}
		\item $2C + \frac{1}{B^2}\geq 0$, \label{inf1}
		\item $\zeta_m<1/3$, \label{inf2}
		\item $\sqrt{3}\zeta_m < 4 + C - s_0/(2-s_0)$, \label{inf3}
		\item $K_0<D$. \label{inf4}
	\end{enumerate}
	Then the solution $Z$ extends to $s=0$.
\end{lem}

\textit{Remark.} The constants $A, B, C$ and $D$ will eventually be chosen by trial and error, depending only on $s_0$ and $Z(s_0)$, taking into consideration the following heuristics.
\begin{itemize}
	\item We choose $A, B>0$ to balance out the contributions from the linear and constant terms of the differential inequality (\ref{gronwallschematic}) in order to get a better Grönwall's estimate. The freedom to choose $A$ and $B$ comes from the estimates (\ref{Aineq}) and (\ref{Bineq}) respectively.
	\item We choose $C$ as small as possible so that the estimate $\zeta_m$ of the right-hand side of (\ref{zetadfn}) remains under $\frac{1}{\sqrt{3}}(4+C+s_0/(2-s_0))$, in order for assumption C to hold throughout $s\in [0, s_0]$.
	\item Since we are assuming that $\zeta_m< 1/3$, we have $K_R\leq K_0$ for all $s$ and consequently assumption D holds so long as we pick $D>K_0$. As $B^2D$ appears in the constant term of (\ref{gronwallschematic}), we choose $D$ to be as small as possible, namely $K_0+\delta$ for some arbitrarily small $\delta>0$ which plays no role in the numerics.
\end{itemize}
\begin{proof}
	Recall that we defined
	\begin{equation*}
		S_{C,D} = \left\{s\in[0,s_0]\; \lvert\; \text{assumptions C and D hold}\right\}
	\end{equation*}
	and 
	\begin{equation}\label{s1dfn}
		s_1 = \sup\{s\in [0,s_0]\;\lvert\; s\notin S_{C,D}\}.
	\end{equation}	
	We estimate using the above considerations and hypothesis (\ref{inf3}) that
	\begin{equation*}
		\abs{Z_1(s_1)}+\abs{Z_2(s_1)}+\abs{Z_3(s_1)}\leq \sqrt{3}\abs{Z(s_1)} \leq 3\zeta_m < 4+C-\frac{s_0}{2-s_0},
	\end{equation*}
	where the second inequality uses Theorem \ref{gronwall}, which is possible as hypothesis (\ref{inf1}) ensures that $\theta\geq 0$. In particular,
	\begin{equation*}
		Z_1(s_1) + Z_2(s_1) + Z_3(s_1) - \left(4-\frac{s_1}{2-s_1}\right)<C.
	\end{equation*}
	Since $\zeta_m<1/3$ by hypothesis (\ref{inf2}), we have
	\begin{equation*}
		R_1(s_1)^2 + R_2(s_1)^2 + R_3(s_1)^2\leq K_R \leq K_0<D.
	\end{equation*}
	By the continuity of $Z_i$ and $R_i$, there is consequently some $\delta>0$ such that $s_1+\delta\in S_{C, D}$, which contradicts (\ref{s1dfn}). Consequently, such an $s_1$ does not exist and $S_{C,D} = [0,s_0]$. In particular, $Z$ and therefore $X$ are bounded up until $s=0$ or $r=1$, which proves the lemma.
\end{proof}

\textit{Remark.} Solutions $Z$ which extend to $s=0$ or equivalently $r=1$ satisfy the boundary condition 
\begin{equation*}
	\frac{d\alpha}{dr}(1) = \alpha(1),\qquad \frac{d\beta}{dr}(1) = \beta(1),\qquad  \frac{d\gamma}{dr}(1) = \gamma(1).
\end{equation*}
We deduce from (\ref{confcompacttransform}) that 
\begin{equation*}
	\lim_{t\to\infty} \frac{\dot a}{a} = \lim_{t\to\infty} \frac{\dot b}{b} = \lim_{t\to\infty} \frac{\dot c}{c} = 1.
\end{equation*}
Consequently, the solutions tend to a ``hyperbolic infinity'' $a\sim a_\infty e^t, b \sim b_\infty e^t, c\sim c_\infty e^t$ for some constants $a_\infty, b_\infty, c_\infty$.

\section{Singular initial value problem and approximate solution}\label{ivpsection}
In this section, we will reformulate the Einstein condition with boundary conditions corresponding to $\mathcal{O}(-4)$ as a first-order singular initial value problem of a form satisfying the hypotheses of \cite[Theorem 4.7]{FH17}. We then linearise the initial value problem around an approximate solution $\hat\eta$.

Consider the tangential equation (\ref{tangentialconformal}) under the change of variables
\begin{equation*}
	X_1=\frac{1}{a}, \qquad X_2=\frac{1}{b},\qquad X_3=\frac{1}{c},\qquad Y_1=\frac{\dot{a}}{a}, \qquad Y_2=\frac{\dot{b}}{b}, \qquad Y_3=\frac{\dot{c}}{c}.
\end{equation*}
In the negatively curved case, it takes the form
\begin{align*}
	\dot{X}_1 &= -X_1Y_1\\
	\dot{Y}_1 &= -Y_1(Y_1 + Y_2 + Y_3) + \frac{1}{2}X_1^2X_2^2X_3^2 \left(X_1^{-4} - (X_2^{-2}-X_3^{-2})^2\right)+3
\end{align*}
and cyclically. Similarly to the equations at conformal infinity, we denote
\begin{equation*}
	R_1 = \frac{a^4 - (b^2-c^2)^2}{2a^2b^2c^2} = \frac{1}{2}X_1^2X_2^2X_3^2 \left(X_1^{-4} - (X_2^{-2}-X_3^{-2})^2\right)
\end{equation*}
and cyclically. Since $a\sim 2t$, $b\sim h+b_1t$ and $c\sim h-b_1t$, we substitute
\begin{gather*}
	X_1 = \frac{1}{2t} + \eta_1, \qquad
	X_2 = \frac{1}{h} - \frac{b_1t}{h^2} + t\eta_2, \qquad
	X_3 = \frac{1}{h} + \frac{b_1t}{h^2} + t\eta_3\\
	Y_1 = \frac{1}{t} + \eta_4,\qquad
	Y_2 = \frac{b_1}{h} + \eta_5, \qquad
	Y_3 = -\frac{b_1}{h} + \eta_6.
\end{gather*}
After some computation, we obtain that the tangential Einstein ODEs (\ref{tangentialconformal}) become a first-order system in terms of $\eta(t) = (\eta_1(t), \eta_2(t), \eta_3(t), \eta_4(t), \eta_5(t), \eta_6(t))^T$ of the form
\begin{equation}\label{etaODEslope2}
	\dot\eta = \frac{1}{t}L\eta + M(t,\eta),
\end{equation}
where $M(t,\eta)=(M_1,M_2,M_3,M_4,M_5,M_6)^T$ and
\begin{equation*}
	L=\begin{pmatrix*}-1 & 0 & 0 & -1/2 & 0 & 0\\
	0 & -1 & 0 & 0 & -1/h & 0\\
	0 & 0 & -1 & 0 & 0 & -1/h\\
	0 & 0 & 0 & -2 & -1 & -1\\
	0 & -h/2 & h/2 & 0 & -1 & 0\\
	0 & h/2 & -h/2 & 0 & 0 & -1\end{pmatrix*},
\end{equation*}
\begin{equation}\label{Mformulas}
	\begin{split}
		M_1&=-\eta_1\eta_4\\
		\qquad M_2&=\frac{b_1^2}{h^3} + \frac{b_1}{h^2}\eta_5 - \frac{b_1}{h}\eta_2 - \eta_2\eta_5\\
		M_3&=\frac{b_1^2}{h^3} - \frac{b_1}{h^2}\eta_6 + \frac{b_1}{h}\eta_3 - \eta_3\eta_6\\
		M_4&=-\eta_4(\eta_4 + \eta_5 + \eta_6) + R_1 + 3\\
		M_5&=-\left(\frac{b_1}{h} + \eta_5\right)(\eta_4 + \eta_5 + \eta_6) + R_2 - \frac{1}{t}\left(\frac{b_1}{h} - \frac{1}{2}h(\eta_2-\eta_3)\right)+3\\
		M_6&=-\left(-\frac{b_1}{h}+\eta_6\right)(\eta_4+\eta_5+\eta_6) + R_3 - \frac{1}{t}\left(-\frac{b_1}{h} + \frac{1}{2}h(\eta_2-\eta_3)\right) + 3.
	\end{split}
\end{equation}

A solution to (\ref{etaODEslope2}) in a neighbourhood of zero with $\eta(0)=0$ corresponds to an Einstein metric with $b(0)=c(0)=h$ and $\dot b(0) = -\dot c(0) = b_1$. Although it does not seem so at first glance, $M(t, \eta)$ is indeed bounded near $t=0$ and continuous in both arguments away from $b,c=0$.

The matrix $L$ has eigenvalues $\{-2,-1,0\}$ and consequently, by \cite[Theorem 4.7]{FH17}, the singular initial value problem consisting of the system (\ref{etaODEslope2}) subject to the initial condition $\eta(0)=0$ is well-posed in a neighbourhood of $t=0$ and there is local existence and uniqueness of solutions.

We will now consider a smooth ``approximate solution'' $\hat{\eta}(t)$ to the above singular initial value problem, and consider the system (\ref{etaODEslope2}) near $\hat{\eta}(t)$. More specifically, it should satisfy
\begin{equation*}
	\hat{E}_2(t) = \frac{1}{t}L\hat{\eta}(t) + M(t, \hat\eta) - \dot{\hat\eta}(t),
\end{equation*}
for some smooth function $\hat{E}_2$ for which we will determine bounds after the fact. We will also require that $\hat\eta(0)=0$. Let $\mu(t) = \eta(t) - \hat\eta(t)$. Then
\begin{align*}
	\dot\mu(t) &= \dot\eta(t) - \dot{\hat{\eta}}(t)\\
	&= \frac{1}{t} L\eta + M(t,\eta) - \frac{1}{t}L\hat\eta - M(t, \hat\eta) + \hat{E}_2\\
	&\equiv \frac{1}{t}L\mu + M_l(t,\hat\eta)\cdot\mu + M_{nl}(t,\hat\eta,\mu) + \hat{E}_2,
\end{align*}
where $M_l$ and $M_{nl}$ designate the linear and nonlinear parts of $M(t,\eta)-M(t,\hat\eta)$ considered as a function of $\mu$. More specifically, since $M(t,\cdot)$ is continuous, we have
\begin{equation*}
	M(t,\eta) - M(t,\hat\eta) = \sum_{i=1}^6 \frac{\partial M}{\partial \eta_i}(t,\hat\eta)\mu_i + \mathcal{O}(\mu^2),
\end{equation*}
so that 
\begin{gather*}
	M_l(t,\hat\eta) = \frac{\partial M}{\partial \vec{\eta}} (t,\hat\eta) = \left(\frac{\partial M_i}{\partial \eta_j}(t,\hat\eta)\right)_{ij}\\
	M_{nl}(t, \hat\eta, \mu) = M(t,\eta) - M(t,\hat\eta) - M_l(t, \hat\eta)\cdot\mu = \mathcal{O}(\mu^2).
\end{gather*}

In order to use a fixed-point argument to show that $\hat\eta$ can be perturbed to a genuine solution, we will eventually need bounds on $M_l(t, \hat\eta)$ and $M_{nl}(t, \hat\eta, \mu)$. We will assume that $C_l, C_{nl}>0$ are constants such that the following hold throughout the subsequent sections, for all $t\in (0, t_f)$:
\begin{gather} 
	\abs{M_l(t, \hat\eta)}\leq C_l\label{Cldfn}\\
	\abs{M_{nl}(t, \hat\eta, \mu)}\leq C_{nl}\abs{\mu(t)}^2 \label{Cnldfn}.
\end{gather}
The constants $C_l$ and $C_{nl}$ depend on the choice of approximate solution $\hat\eta$ and will be determined later using the expressions in appendix \ref{MlMnlsection}.

\section{Proof of existence : fixed-point argument starting from the singular orbit}\label{fixedpointsection}

\subsection{$C^1$-type estimate}

Consider the linear inhomogeneous initial value problem
\begin{equation}\label{firstinhom}
	\dot\mu(t) = \frac{1}{t}L\mu + F(t)
\end{equation}
subject to the initial conditions $\mu(0)=0$. It has a unique solution, up to some time which we will denote $t_f$. Since the equation is linear, we may define a solution map $\mathcal{L}$ by
\begin{equation*}
	\mu = \mathcal{L} F.
\end{equation*}
We will prove (part of) the equivalent of \cite[Lemma 13]{BH26}. The proof is mostly identical, but we reproduce it here for completeness.
\begin{lem}\label{C1lemma}
	For each $t>0$, we have the estimate $\abs{(\mathcal{L}F)(t)/t} \leq B\norm{F}_{C^0((0,t))}$, where
	\begin{equation*}
		B = \frac{19}{8} + \frac{1}{e} + \frac{1}{2}\left(h + \frac{1}{h}\right).
	\end{equation*}
\end{lem}
\begin{proof}
	We note that (\ref{firstinhom}) can be written
	\begin{equation*}
		\frac{d}{dt}(\exp(-L\log t)\mu(t)) = \exp(-L\log t) F(t),
	\end{equation*}
	so that for each $0<t_0<t$,
	\begin{equation}\label{preBequation}
		\mu(t) - \exp(L(\log t - \log t_0))\mu(t_0) = \exp(L\log t)\int_{t_0}^t\exp(-L\log s)F(s)\, ds.
	\end{equation}
	Consider the prefactor of $\mu(t_0)$, which we denote
	\begin{equation*}
		\mathcal{B}\equiv\exp(L(\log t - \log t_0)).
	\end{equation*}
	We would like to show that the operator norm (denoted $\norm{\cdot}$) of $\mathcal{B}$ is bounded by a constant independent of $t$ and $t_0$, in order to eventually take the limit of (\ref{preBequation}) as $t_0\to 0$. For simplicity, we split $\mathcal{B}$ into submatrices
	\begin{equation*}
		\mathcal{B} = \begin{pmatrix*} B_1 & B_2 \\ B_3 & B_4\end{pmatrix*},
	\end{equation*}
	where
	\begin{gather*}
		B_1 = \begin{pmatrix*}
			u & 0 & 0\\[0.1cm]
			0 & \frac{1}{4}(u+1)^2 & -\frac{1}{4}(u-1)^2\\[0.1cm]
			0 & -\frac{1}{4}(u-1)^2 & \frac{1}{4}(u+1)^2
		\end{pmatrix*},\\
		B_2 = \begin{pmatrix*}\frac{1}{2}(u^2-u) & \frac{1}{2}\left(u^2-u+u\log\frac{1}{u}\right) & \frac{1}{2}\left(u^2-u+u\log\frac{1}{u}\right)\\[0.1cm]
			0 & \frac{1}{4h}\left(u^2 - 1 - 2u\log \frac{1}{u}\right) & \frac{1}{4h}\left(1-u^2-2u\log\frac{1}{u}\right)\\[0.1cm]
			0 & \frac{1}{4h}\left(1-u^2-2u\log\frac{1}{u}\right) & \frac{1}{4h}\left(u^2 - 1 - 2u\log \frac{1}{u}\right)\end{pmatrix*},\\
		B_3 = \frac{h}{4}\begin{pmatrix*}0 & 0 & 0\\
			0 & u^2-1 & 1-u^2\\
			0 & 1-u^2 & u^2-1
			\end{pmatrix*}, \quad B_4 = \begin{pmatrix*}u^2 & u^2-u & u^2-u\\[0.1cm]
			0 & \frac{1}{4}(u+1)^2 & -\frac{1}{4}(u-1)^2\\[0.1cm]
			0 & -\frac{1}{4}(u-1)^2 & \frac{1}{4}(u+1)^2\end{pmatrix*},
	\end{gather*}
	and $u=t_0/t$, so that $0<u<1$. We compute
	\begin{equation*}
		\norm{B_1} = \frac{1}{2}(1+u^2), \qquad
		\norm{B_3} = \frac{h}{2}(1-u^2), \qquad
		\max_{u\in[0,1]}\norm{B_4} \leq 1.
	\end{equation*}
	Let
	\begin{equation*}
		\tilde{B_2} = \begin{pmatrix*}\frac{1}{4h}\left(u^2 - 1 - 2u\log \frac{1}{u}\right) & \frac{1}{4h}\left(1-u^2-2u\log\frac{1}{u}\right)\\[0.1cm]
			\frac{1}{4h}\left(1-u^2-2u\log\frac{1}{u}\right) & \frac{1}{4h}\left(u^2 - 1 - 2u\log \frac{1}{u}\right)
		\end{pmatrix*}.
	\end{equation*}
	It satisfies
	\begin{equation*}
		\norm{\tilde{B_2}} = \max\left\{\frac{1}{2h}(1-u^2), \frac{1}{h}u\log\frac{1}{u}\right\},
	\end{equation*}
	which implies that, since $u\log \frac{1}{u}\leq \frac{1}{e}$ for all $u\in[0,1]$,
	\begin{equation*}
		\max_{u\in [0,1]} \norm{\tilde{B_2}} \leq \frac{1}{2h}.
	\end{equation*}
	Finally, since $\frac{1}{2}(u^2-u)\leq \frac{1}{8}$,
	\begin{equation*}
		\norm{B_2}\leq \frac{3}{8} + \frac{1}{e} + \norm{\tilde{B_2}} \leq \frac{3}{8} + \frac{1}{e} + \frac{1}{2h}.
	\end{equation*}
	We thus have, for any $0<u<1$, that
	\begin{equation*}
		\norm{\mathcal{B}}\leq \norm{B_1} + \norm{B_2} + \norm{B_3} + \norm{B_4} \leq \frac{19}{8} + \frac{1}{e} + \frac{1}{2}\left(h + \frac{1}{h}\right)\equiv B.
	\end{equation*}
	
	Consequently, since $\mu(0)=0$, we can take the limit $t_0\to 0$ in (\ref{preBequation}) to obtain
	\begin{equation*}
		\mu(t) = \int_0^t \exp(L(\log t - \log s)) F(s)\, ds.
	\end{equation*}
	Then, we have that
	\begin{equation*}
		\abs{\mu(t)} \leq \int_0^t B\abs{F(s)}\, ds \leq Bt_f\norm{F}_{C^0((0,t))},
	\end{equation*}
	as well as
	\begin{equation*}
		\left\vert\frac{\mu(t)}{t}\right\vert\leq B\norm{F}_{C^0((0,t))}.
	\end{equation*}
	
\end{proof}

\subsection{Linear inhomogeneous initial value problem and $C^0$ estimate}

Now we consider the inhomogeneous problem which includes the linear term, namely 
\begin{equation}\label{inhomog_linear}
	\dot\mu(t) = \frac{1}{t}L\mu + M_l(t,\hat\eta)\mu + F_2(t).
\end{equation}
It can be written as
\begin{equation*}
	\mu = \mathcal{L}(M_l(t,\hat\eta)\mu + F_2(t))
\end{equation*}
with, in the notation of the previous subsection, $F=M_l(t,\hat\eta)\mu + F_2(t)$. We then obtain by (\ref{Cldfn}) and Lemma \ref{C1lemma} the estimate
\begin{equation}\label{C1typeest}
	\left\vert\frac{\mu(t)}{t}\right\vert\leq B \left(C_l \norm{\mu}_{C^0((0,t))} + \norm{F_2}_{C^0((0,t))}\right).
\end{equation}

\subsubsection{Non-negativity of $L$}

We note some important properties of the matrix $L$. The characteristic polynomial of its symmetrization $L_{sym} = \frac{1}{2}(L+L^T)$ is given by
\begin{multline*}
	p_{sym}(x) = \left(x-\frac{(h-1)^2}{2h}\right)\left(x+\frac{(h+1)^2}{2h}\right)\\
	\times \left(x^4 + 5x^3 + \left(\frac{135}{16} - \frac{1}{4h^2}\right)x^2 + \left(\frac{47}{8} - \frac{3}{4h^2}\right)x + \frac{23}{16} - \frac{31}{64h^2}\right).
\end{multline*}
We will assume throughout the remainder of this work that $h>\frac{1}{2}\sqrt{\frac{31}{23}}$ (noting that the boundary of the shaded region of Figure \ref{b1hplot} at $b_1=0$ appears to be at $\frac{1}{2}\sqrt{\frac{31}{23}}\approx 0.58$). Then, all the coefficients of $p_{sym}$ are positive and by Descartes' rule of signs $L_{sym}$ has precisely one positive eigenvalue,
\begin{equation*}
	\lambda\equiv \frac{(h-1)^2}{2h}.
\end{equation*}

\subsubsection{$C^0$ estimate}
We would like to estimate the quantity $\frac{2}{t}\mu^T L\mu$. First, we note that $\mu^T L\mu = \mu^T L_{sym}\mu \leq \lambda \abs{\mu}^2$. An issue is that the prefactor $2/t$ blows up at $t=0$, so in a region surrounding $t=0$ we must instead use the ``$C^1$-type estimate'' (\ref{C1typeest}). To do so, we break the time domain into two parts, delimited by $0<t_0<t_f$, and write
\begin{equation*}
	\frac{2}{t}\lambda\abs{\mu}^2\leq h(t) + r(t)\abs{\mu}^2,
\end{equation*}
where $h(t)$ is some function vanishing outside of $(0,t_0)$, and $r(t)$ is the continuous and uniformly bounded function
\begin{equation*}
	r(t)=\begin{cases*}
		\frac{2}{t_0}\lambda & if $0\leq t<t_0$\\
		\frac{2}{t}\lambda & if $t\geq t_0$.
	\end{cases*}
\end{equation*}
In practice, $t_0$ will be chosen to be much smaller than $t_f$. To choose a suitable function $h$, note first that
\begin{equation*}
	\frac{2}{t} - \frac{2}{t_0} = 2\left(1-\frac{t}{t_0}\right)\frac{1}{t}
\end{equation*}
and that $(a+b)^2\leq 2a^2+2b^2$, so we choose, using (\ref{C1typeest}),
\begin{equation}\label{hdfn}
	h(t)=\begin{cases*}2t\left(1-\frac{t}{t_0}\right)\cdot 2\lambda B^2\left(C_l^2\norm{\mu}_{C^0((0,t))}^2 + \norm{F_2}_{C^0((0,t))}^2\right) & if $0\leq t<t_0$\\
		0 & if $t\geq t_0$.
		\end{cases*}
\end{equation}
We then estimate the solution to (\ref{inhomog_linear}) by
\begin{align*}
	2\mu\cdot \dot\mu = (\abs{\mu}^2)^\cdot &= \frac{2}{t}\mu^T L \mu + 2\mu^T M_l(t,\hat\eta)\mu + 2\mu \cdot F_2(t)\\
	&\leq \frac{2}{t}\lambda\abs{\mu}^2 + 2C_l\abs{\mu}^2 + 2\abs{\mu}\abs{F_2}\\
	&\leq \left(r(t) + 2C_l +1\right)\abs{\mu(t)}^2 + \abs{F_2(t)}^2 + h(t).
\end{align*}
Another version of Grönwall's inequality states the following.
\begin{lem}
Let $\alpha, \beta$ be continuous real functions on the interval $[0,T]$ with $\alpha\geq 0$. If $y\geq 0$ is a differentiable function on $[0, T]$ satisfying $y(0)=0$ and
\begin{equation*}
	\dot y\leq \alpha y + \beta,
\end{equation*}
then, for each $t\in [0, T)$,
\begin{equation*}
	y(t)\leq \mathcal{I}(t) \int_0^t \mathcal{I}(s)^{-1}\beta(s)\, ds,
\end{equation*}
where
\begin{equation*}
	\mathcal{I}(t)=\exp\left(\int_0^t\alpha(s)\, ds\right) \geq 1.
\end{equation*}
\end{lem}

Applying the inequality to our problem, we have
\begin{equation*}
	\abs{\mu(t)}^2\leq \mathcal{I}(t)\int_0^t \abs{F_2(s)}^2 + h(s)\, ds.
\end{equation*}
In particular, recalling that $h$ is supported in $(0, t_0)$ and continuing until the end of the time interval $(0, t_f)$,
\begin{equation*}
	\norm{\mu}_{C^0}^2\leq \mathcal{I}(t_f)\left(\norm{F_2}_{L^2}^2 + \int_0^{t_0}h(s)\, ds\right),
\end{equation*}
where the norms $C^0$ and $L^2$ are over the entirety of $(0,t_f)$. Integrating the prefactor of $h(t)$ over its support yields
\begin{equation*}
	\int_0^{t_0}2t\left(1-\frac{t}{t_0}\right) dt = \frac{t_0^2}{3},
\end{equation*}
so we obtain
\begin{equation}\label{gronwalloutput}
	\norm{\mu}_{C^0}^2\leq \mathcal{I}(t_f)\left(\norm{F_2}_{L^2}^2 + \frac{t_0^2}{3}\cdot 2\lambda B^2(C_l^2\norm{\mu}_{C^0}^2 + \norm{F_2}_{C^0}^2)\right).
\end{equation}
An explicit computation yields
\begin{equation*}
	\mathcal{I}(t_f)=\exp((2C_l+1)t_f + 2\lambda)\left(\frac{t_f}{t_0}\right)^{2\lambda}.
\end{equation*}
The prefactor of $\norm{\mu}_{C^0}^2$ on the right-hand side of (\ref{gronwalloutput}) is then
\begin{equation*}
	\frac{2}{3}e^{(2C_l+1)t_f + 2\lambda}\, t_f^{2\lambda}\lambda B^2 C_l^2 t_0^{2-2\lambda}.
\end{equation*}
We set this equal to $\frac{1}{2}$, and choose $t_0$ to be the corresponding value, namely
\begin{equation*}
	t_0 = \left(\frac{\sqrt{3}}{2} \frac{e^{-(C_l+1/2)t_f-\lambda}}{t_f^\lambda BC_l\sqrt{\lambda}}\right)^{1/(1-\lambda)}.
\end{equation*}
Note that this is only possible if $\lambda<1$, or equivalently
\begin{equation}\label{hconstraint}
	h\in (2-\sqrt{3}, 2+\sqrt{3}).
\end{equation}
We consequently have the estimate
\begin{equation*}
	\frac{1}{2}\norm{\mu}_{C^0}^2 \leq \mathcal{I}(t_f)\left(\norm{F_2}_{L^2}^2 + \frac{2}{3}t_0^2B^2\norm{F_2}_{C^0}^2\right).
\end{equation*}
To make the following steps simpler, we loosen the estimate slightly to
\begin{equation*}
	\norm{\mu}_{C^0}\leq \sqrt{\mathcal{I}(t_f)}\left(\sqrt{2}\norm{F_2}_{L^2} + \frac{2}{\sqrt{3}}t_0B\norm{F_2}_{C^0}\right).
\end{equation*}
We will eventually estimate $\norm{F_2}_{C^0}$ directly, so we can further weaken the above to
\begin{equation*}
	\norm{\mu}_{C^0}\leq \sqrt{\mathcal{I}(t_f)}\left(\sqrt{2t_f} + \frac{2}{\sqrt{3}}t_0B\right)\norm{F_2}_{C^0}.
\end{equation*}
We summarise the results as follows. Let $\mathcal{S}_2 : C^0((0,t_f)) \rightarrow C^0((0,t_f))$ be the linear solution map defined by 
\begin{equation}\label{S2dfn}
	\mu = \mathcal{S}_2 F_2,
\end{equation}
where $\mu$ is the unique solution to the linear equation (\ref{inhomog_linear}), which we assume \textit{a priori} to exist up until time $t_f$. Then, we have shown the following.

\begin{lem}
	The operator $\mathcal{S}_2$ defined by (\ref{S2dfn}) is bounded with norm at most
	\begin{equation*}
		\sqrt{\mathcal{I}(t_f)}\left(\sqrt{2t_f} + \frac{2}{\sqrt{3}}t_0B\right).
	\end{equation*}
\end{lem}

\subsection{Fixed-point theorem and proof of existence}

Consider now a true solution $\hat\eta + \mu$ of the Einstein equation corresponding to the initial data $h, b_1$. Then, it satisfies $\mu(0)=0$ and
\begin{equation}\label{mutrueequation}
	\dot\mu=\frac{1}{t}L\mu + M_l(t,\hat\eta)\cdot\mu + M_{nl}(t,\hat\eta,\mu) + \hat{E}_2.
\end{equation}
Correspondingly, on the interval $(0,t_f)$, we have
\begin{equation*}
	\mu = \mathcal{S}_2 (\hat{E}_2 + M_{nl}(t,\hat\eta, \mu)).
\end{equation*}
Suppose that $\norm{\hat E_2}_{C^0}<\varepsilon$. Then, we get using (\ref{Cnldfn}) the estimate
\begin{equation*}
	\norm{\mu}_{C^0}\leq \sqrt{\mathcal{I}(t_f)}\left(\sqrt{2t_f}+\frac{2}{\sqrt{3}}t_0B\right)\left(\varepsilon + C_{nl}\norm{\mu}_{C^0}^2\right).
\end{equation*}
We can consequently apply a version of the Schauder fixed point theorem to construct a solution $\mu$. More specifically, we will use the following special case of \cite[Theorem 2]{BH26}.
\begin{thm}\label{fixedpointtheorem}
	Let $X$ be a Banach space and $Q: X\rightarrow X$ be a map satisfying $\norm{Q(x)}\leq q\norm{x}^2$ whenever $\norm{x}\leq s_0$. If $\norm{x_0} + qs^2<s$ for some $x_0\in X$ and $0<s\leq s_0$, and $Q$ is completely continuous, then the fixed point equation
	\begin{equation*}
		x=x_0 + Q(x)
	\end{equation*}
	has a solution $x^*\in X$ satisfying $\norm{x^*}\leq s$.
\end{thm}

To apply the above theorem in our context, we work in $X = C^0((0, t_f))^6$, the space of bounded continuous functions $(0,t_f)\rightarrow \mathbb{R}^6$, which is a Banach space. We also have $x=\mu$ and $x_0=\mathcal{S}_2\hat{E}_2$. Consequently, the above estimates give
\begin{equation*}
	\norm{x_0}\leq \sqrt{\mathcal{I}(t_f)}\left(\sqrt{2t_f} + \frac{2}{\sqrt{3}}t_0B\right)\varepsilon,
\end{equation*}
Similarly, we have $Q(x) = \mathcal{S}_2 M_{nl}$, giving the estimate
\begin{equation*}
	\norm{Q(x)}\leq q\norm{\mu}^2,
\end{equation*}
where
\begin{equation*}
	q=\sqrt{\mathcal{I}(t_f)}\left(\sqrt{2t_f}+\frac{2}{\sqrt{3}}t_0B\right)C_{nl}.
\end{equation*}
The existence of an $s$ such that $\norm{x_0}+qs^2<s$ is equivalent to
\begin{equation*}
	\norm{x_0}<\frac{1}{4q},
\end{equation*}
so we can choose $s$ to be equal to $1/(2q)$. We need then that
\begin{equation*}
	\varepsilon<\frac{1}{4\mathcal{I}(t_f)C_{nl}(\sqrt{2t_f}+\frac{2}{\sqrt{3}}t_0B)^2}\equiv \varepsilon_0.
\end{equation*}
Consequently, we have shown the following.

\begin{lem}\label{existencelemma}
	Suppose that there exists a smooth approximate solution $\hat\eta(t)$ on $[0,t_f]$ satisfying
	\begin{enumerate}
		\item $\hat\eta(0)=0$,\label{hyp1}
		\item $\abs{M_l(t,\hat\eta)}\leq C_l$ for each $t\in [0, t_f]$,\label{hyp2}
		\item $\abs{M_{nl}(t, \hat\eta, \mu)}<C_{nl}\abs{\mu(t)}^2$ for each $t\in[0,t_f]$,\label{hyp3}
		\item $\norm{\hat E_2}_{C^0}<\varepsilon_0$, where $\hat{E}_2(t) = \frac{1}{t}L\hat{\eta}(t) + M(t, \hat\eta) - \dot{\hat\eta}(t)$, $\varepsilon_0=(4\mathcal{I}(t_f)C_{nl}(\sqrt{2t_f}+\frac{2}{\sqrt{3}}t_0B)^2)^{-1}$ and $\mathcal{I}(t_f)$, $t_0$, $B$ are given as above.\label{hyp4}
	\end{enumerate}
	Then, there exists a solution $\mu$ to the fixed point problem $\mu = \mathcal{S}_2(\hat{E}_2 + M_{nl}(t,\hat\eta, \mu))$, where $\mathcal{S}_2$ is given by (\ref{S2dfn}), and consequently $\eta = \mu + \hat\eta$ is an Einstein metric on $M$ defined for $t\in [0,t_f]$. Furthermore, 
	\begin{equation}\label{mutfbounds}
		\abs{\mu(t)}\leq \frac{1}{2\sqrt{\mathcal{I}(t_f)}\left(\sqrt{2t_f}+\frac{2}{\sqrt{3}}t_0B\right)C_{nl}}
	\end{equation}
	for each $t$, in particular for $t=t_f$.
\end{lem}

\section{Computer-assisted construction}\label{computersection}

In this section, we construct the smooth function $\hat\eta$ that satisfies the hypotheses of Lemma \ref{existencelemma} and can be perturbed to a metric at $t_f$ satisfying the hypotheses of Lemma \ref{infinitylemma}, using a procedure similar to that used in \cite{BH26}, which goes as follows.
\begin{enumerate}
	\item Choosing a point $(h, b_1)$ of Figure \ref{b1hplot} where there is numerical evidence of a complete Einstein metric, construct a heuristic solution up to arbitrary precision, with no guarantees of accuracy. This is done using power series methods, in particular using Frobenius' method near the singular orbit, as discrete-time iterative schemes such as Runge--Kutta methods are unable to reach the required levels of precision. \label{heuristicstep}
	\item Approximate the heuristic solution of step (\ref{heuristicstep}) using Chebyshev polynomials satisfying hypothesis (\ref{hyp1}) of Lemma \ref{existencelemma} whose coefficients are rigorously controlled with interval arithmetic. \label{chebyshevstep}
	\item Writing $M_l$ and $M_{nl}$ as rational functions in terms of the Chebyshev polynomial approximations of step \ref{chebyshevstep}, obtain rigorous bounds $C_l$, $C_{nl}$ satisfying hypotheses (\ref{hyp2}) and (\ref{hyp3}) of Lemma \ref{existencelemma} using Sturm's theorem. \label{MlMnlstep}
	\item Compute $\varepsilon_0$ using interval arithmetic from $C_l$, $C_{nl}$ and $t_f$, and compare to a good estimate $\varepsilon$ of $\norm{\hat E_2}_{C^0}$ as a function of the Chebyshev approximation, obtained using Sturm's theorem. If $\varepsilon<\varepsilon_0$, then hypothesis (\ref{hyp4}) holds and so does the conclusion of Lemma \ref{existencelemma}. We consequently have a smooth Einstein metric $\eta$ on $M$ from the singular orbit to $t_f$. \label{fixedpointstep}
	\item Letting $s_0 = 1-\tanh(t_f/2)$, find constants $A, B, C, D$ such that the hypotheses of Lemma \ref{infinitylemma} hold for the initial data $Z(s_0)$ corresponding to the solution $\eta$ of step (\ref{fixedpointstep}). The lemma then gives us that $Z$ extends to $s=0$, or equivalently $\eta$ extends to $t\to \infty$, giving us a complete Einstein metric on $M$. By the well-posedness of the ODEs and the continuity of the arguments in \S \ref{infinitysection}, we also conclude that there is a (2-dimensional) neighbourhood of $(h, b_1)$ which is in correspondence with distinct complete Einstein metrics. \label{finalstep}
\end{enumerate}

We may choose essentially any point in the dark shaded region of Figure \ref{b1hplot} satisfying (\ref{hconstraint}). For convenience, we will pick
\begin{equation*}
	h=1.5, \quad b_1 = 0.1.
\end{equation*}
The specific values of the various constants and estimates used in steps (\ref{MlMnlstep})-(\ref{finalstep}) are listed in appendix \ref{valuesapp}. We perform rigorous numerical computations using arbitrary precision interval arithmetic via the software package \texttt{arb} \cite{J17}. All implementations were done in Python, and are available at \url{https://github.com/Qiu-Shi-Wang/SU2-Einstein}.
\subsection{Formal power series solutions}
Near $t=0$, we must use Frobenius' method. Consider the singular initial value problem
\begin{equation*}
	\dot\eta = \frac{1}{t}L\eta + M(t, \eta), \qquad \eta(0)=0,
\end{equation*}
for an analytic function $M$.\footnote{Although it is not immediately apparent, for each $t\geq 0$, $M(t, \cdot)$ is in fact analytic near the approximate solution $\hat\eta(t)$. This can be seen in the expressions for $\frac{\partial M_i}{\partial \eta_j}$ at the end of \S\ref{jacobiansubsection}. Hence, we can indeed use power series methods. Indeed, we know \textit{a priori} \cite[Theorem 5.26]{B87} that solutions to the Einstein equation are analytic, as a result of its ellipticity.} Expand the solution near $t=0$ as the series $\eta = \sum_{i=1}a_it^i$. Then
\begin{equation*}
	\dot\eta = \sum_{i=0} a_{i+1}(i+1) t^i, \qquad \frac{1}{t}L\eta = \sum_{i=0}La_{i+1}t^i,
\end{equation*}
and we consequently get the recurrence relation
\begin{equation*}
	\sum_{i=0}\left((i+1)\mathrm{id}_6 + L\right) a_{i+1}t^i = M(t, \eta).
\end{equation*}
Expanding $M(t, \eta) = \sum_{i=0} b_i t^i$, where we note that $b_i$ depends only on $a_0, \dots, a_i$, we obtain the iterative formula
\begin{equation*}
	a_{i+1} = \left((i+1)\mathrm{id}_6 - L\right)^{-1}b_i.
\end{equation*}

Away from $t=0$, we can centre the power series around some other time $t_0>0$ and obtain a similar recurrence relation. More specifically, write $\tau = t-t_0$, and let
\begin{equation*}
	\eta(\tau + t_0) = \sum_{i=0} a_i \tau^i,
\end{equation*}
where we obtain the value $a_0 = \eta(t_0)$ from evaluation of the previous power series. Let
\begin{equation*}
	M(t, \eta) = M(\tau + t_0, \eta) = \sum_{i=0} b_i \tau^i.
\end{equation*}
The Einstein equation is then given by
\begin{equation*}
	\sum_{i=1} a_i i \tau^{i-1} = \frac{1}{\tau+t_0}\sum_{i=0} La_i \tau^i + \sum_{i=0} b_i \tau^i.
\end{equation*}
We obtain from it the iterative formula
\begin{equation*}
	a_{i+1} = \frac{1}{i+1}\left(\frac{1}{t_0}\gamma_i + b_i\right),
\end{equation*}
where
\begin{equation*}
	\gamma_i = \sum_{l=0}^i \alpha_l \beta_{i-l}, \qquad \alpha_i = La_i, \qquad \beta_i = \left(-\frac{1}{t_0}\right)^i.
\end{equation*}

We will estimate the radius of convergence of each formal power series using the root test. More specifically, given a (vector-valued) series $\sum_k c_k t^k$, the radius of convergence $r$ is given by $r^{-1} = \limsup_{k\to\infty} \norm{c_k}^{1/k}$. We estimate this quantity by taking the maximum of the 10 ``final terms'' of the lim sup. More specifically, given a formal power series computed up to $N$ terms, we approximate its radius of convergence by
\begin{equation*}
	\hat r = \left(\max_{i=0, ..., 9} \|c_{N-i}\|^{\frac{1}{N-i}}\right)^{-1}.
\end{equation*}

Each power series is then evaluated at $\hat r/2$ past its centre, at which point another formal power series solution is computed. This procedure is repeated until the final time $t_f$ is reached, resulting in a heuristic approximate solution $\tilde\eta(t)$ for $t\in [0, t_f]$.

\subsection{Chebyshev polynomial approximations}
We will recall that the \textit{Chebyshev polynomials of the first kind} $T_n$, for $n\in \mathbb{N}$, are defined on $[-1,1]$ through the identity
\begin{equation*}
	T_n(\cos\theta) = \cos(n\theta).
\end{equation*}
$T_n(x)$ has zeroes $\{x_k\}$ for $k=1, \dots, n$, where 
\begin{equation*}
	x_k = \cos\left(\frac{\pi(k-\frac{1}{2})}{n}\right).
\end{equation*}
We compute the heuristic derivative of $\tilde\eta(t)$ using the right-hand side:
\begin{equation*}
	\tilde\eta_D(t)\equiv \frac{1}{t}L\tilde\eta + M(t, \tilde\eta).
\end{equation*}
Consider a function $f$ on $[-1,1]$. If we let
\begin{align*}
	c_j &= \frac{2}{N}\sum_{k=1}^N f(x_k)T_j(x_k)\\
	&= \frac{2}{N}\sum_{k=1}^N f\left(\cos\left(\frac{\pi\left(k-\frac{1}{2}\right)}{N}\right)\right)\cos\left(\frac{\pi j \left(k-\frac{1}{2}\right)}{N}\right),
\end{align*}
then we may approximate $f$ by a degree $N-1$ polynomial via \cite[Theorem 6.7]{MH03}
\begin{equation*}
	f(x)\approx \sum_{j=0}^{N-1} c_j T_j(x) - \frac{1}{2}c_0.
\end{equation*}
The approximation is exact at the $N$ zeroes of $T_N(x)$, which follows from the fact that the $T_j$'s satisfy on the zero set $\{x_k\}$ of $T_{n+1}$ the discrete orthogonality property \cite[\S 4.6]{MH03}
\begin{equation*}
	\sum_{k=1}^{n+1}T_i(x_k)T_j(x_k) = \begin{cases}
	0 & \text{if $i\neq j$ ($\leq n$)}\\
	n+1 & \text{if $i=j=0$}\\
	\frac{1}{2}(n+1) & \text{if $0<i=j\leq n$.}\end{cases}
\end{equation*}

We approximate $\tilde\eta_D$ by a sum of Chebyshev polynomials $\hat\eta_D(t)$ in exactly this manner, after rescaling the $T_j$'s to be defined on $[0,t_f]$ rather than $[-1,1]$.

We can then integrate $\hat\eta_D(t)$ starting from $t=0$, using the identity
\begin{equation*}
	\int T_n\, dt = \frac{1}{2(n+1)}T_{n+1} - \frac{1}{2(n-1)}T_{n-1},
\end{equation*}
to obtain an approximate solution $\hat\eta(t)$ satisfying $\hat\eta(0)=0$ (and thus hypothesis (\ref{hyp1}) of Lemma \ref{existencelemma}) as a sum of Chebyshev polynomials.

\subsection{Estimates on the linear and nonlinear terms}
Using the expressions in appendix \ref{MlMnlsection}, we note that $M_l(t, \hat\eta)$ and $M_{nl}(t,\hat\eta, \mu)$ can be written as rational functions of $\eta$. Using the identity
\begin{equation*}
	T_m T_n = \frac{1}{2}\left(T_{m+n} - T_{\abs{m-n}}\right),
\end{equation*}
we can write each term of $M_l$ and $M_{nl}$ as a rational function of finite sums of Chebyshev polynomials with coefficients controlled by interval arithmetic. They can in turn be converted into a quotient of polynomials $P(x)/Q(x)$ in the usual monomial basis.

We note that $\abs{P(x)/Q(x)}<\varepsilon$ for all $x\in [a,b]$ if:
\begin{itemize}
	\item There exists an $x_0\in (a,b)$ such that $\abs{P(x_0)/Q(x_0)}<\varepsilon$,
	\item $P-\varepsilon Q$ and $P+\varepsilon Q$ have no zeroes in $[a,b]$.
\end{itemize}
In order to prove that a polynomial with coefficients given in interval arithmetic form has no zeroes, we use Sturm's theorem:
\begin{thm}
	Let $P(x)$ be a polynomial with real coefficients. Define the \textit{Sturm sequence} $P_i$, $i\geq 0$, by
	\begin{equation*}
		P_0 = P,\qquad P_1 = P', \qquad P_{i+1}=-\mathrm{rem}(P_{i-1},P_i),
	\end{equation*}
	where $\mathrm{rem}$ is the remainder of polynomial division. There are at most $\deg P$ nonzero terms. For any $\xi\in\mathbb{R}$, consider the sequence
	\begin{equation*}
		P_0(\xi), P_1(\xi), P_2(\xi), \dots.
	\end{equation*}
	We will denote the number of \textit{sign variations} in this sequence, disregarding zeroes, by $V(\xi)$.
	Then, $V(a) - V(b)$ is the number of distinct real roots of $P$ in $(a,b]$.
\end{thm}
We use binary search to determine reasonably sharp bounds $\varepsilon$ on each relevant expression. This procedure produces rigorous bounds $C_l$ and $C_{nl}$ satisfying hypotheses (\ref{hyp2}) and (\ref{hyp3}) of Lemma \ref{existencelemma} respectively. It is worth noting that the $\mathcal{O}(\mu^3)$ terms of $M_{nl}$ can be disregarded, at the cost of a negligible fractional increase in $C_{nl}$, as the perturbation $\mu$ will be very small and consequently all terms will be dominated by the leading order terms.

\textit{Remark. } To reduce the numerical precision required in the Euclidean divisions, similarly to \cite{BH26}, we rescale the variable $x$ of $P(x)$ to $y=x/\rho$, and consider the sign variations of $\tilde P(y) = \tilde P(x/\rho) = P(x)$ in $(\rho a, \rho b]$. Practically, we notice that $\rho=0.5$ is a convenient choice for our problem.

\subsection{Smooth metric up until $t_f$}

Using the method of the previous subsection, we can obtain a $C^0$ estimate on
\begin{equation*}
	\hat E_2 = \frac{1}{t}L\hat{\eta}(t) + M(t, \hat\eta) - \hat\eta_D(t),
\end{equation*}
which we denote by $\varepsilon$. Then, we compute
\begin{equation*}
	\varepsilon_0 = \frac{1}{4\mathcal{I}(t_f)C_{nl}(\sqrt{2t_f}+\frac{2}{\sqrt{3}}t_0B)^2}
\end{equation*}
using interval arithmetic. Here, recall that
\begin{gather*}
	\mathcal{I}(t_f)=\exp((2C_l+1)t_f + 2\lambda)\left(\frac{t_f}{t_0}\right)^{2\lambda}, \qquad \lambda = \frac{(h-1)^2}{2h},\\
	B = \frac{19}{8} + \frac{1}{e} + \frac{1}{2}\left(h + \frac{1}{h}\right), \qquad t_0=\left(\frac{\sqrt{3}}{2} \frac{e^{-(C_l+1/2)t_f-\lambda}}{t_f^\lambda\sqrt{\lambda} BC_l}\right)^{1/(1-\lambda)}.
\end{gather*}
If $\varepsilon<\varepsilon_0$, then hypothesis (\ref{hyp4}) of Lemma \ref{existencelemma} holds. The lemma consequently stipulates the existence of a real solution $\mu$ up to time $t_f$ with bounds (\ref{mutfbounds}).

\subsection{Extension to infinity}\label{extensionsubsection}
The procedure in this subsection follows \S\ref{infinitysection}. Given a choice of constants $A, B, C, D$, we can use the bound (\ref{mutfbounds}) to give a rigorous estimate of $\abs{Z_0}^2$ from above, and consequently estimate from above the quantity $\zeta_m$ of \S\ref{CDsubsection} using interval arithmetic. We can also get upper estimates for $K_0$ from the bound (\ref{mutfbounds}).

More precisely, consider the solution $\eta$ produced by Lemma \ref{existencelemma} at the final time $t_f$. It is given by
\begin{equation*}
	\eta(t_f) = \hat\eta(t_f) + \mu(t_f),
\end{equation*}
where the bound (\ref{mutfbounds}) on $\mu(t_f)$ can be computationally absorbed into the interval-arithmetic error of $\eta(t_f)$. We can then compute $\abs{Z_0} = \sqrt{Z_1(s_0)^2 + Z_2(s_0)^2 + Z_3(s_0)^2}$, where
\begin{equation}\label{Zs0expression}
	\begin{split}
		Z_1(s_0) &= 1 + \frac{1}{\rho_0}\left(r_0 - \frac{1}{t_f} - \eta_4(t_f)\right)\\
		Z_2(s_0) &= 1 + \frac{1}{\rho_0}\left(r_0 - \frac{b_1}{h} - \eta_5(t_f)\right)\\
		Z_3(s_0) &= 1 + \frac{1}{\rho_0}\left(r_0 + \frac{b_1}{h} - \eta_6(t_f)\right),
	\end{split}
\end{equation}
and $r_0 = \tanh(t_f/2)$, $\rho_0 = \frac{1}{2}(1-r_0^2)$. Similarly, we can obtain rigorous estimates on 
\begin{equation*}
	K_0 = \frac{3}{8\gamma_0^4}\left(\frac{\alpha_0^4}{\beta_0^4} + \frac{\beta_0^4}{\alpha_0^4} \right) + \frac{3}{8\beta_0^4}\left(\frac{\alpha_0^4}{\gamma_0^4} + \frac{\gamma_0^4}{\alpha_0^4}\right) + \frac{3}{8\alpha_0^4}\left(\frac{\beta_0^4}{\gamma_0^4} + \frac{\gamma_0^4}{\beta_0^4} \right),
\end{equation*}
where
\begin{align*}
	\alpha_0 &= \frac{\rho_0}{\frac{1}{2t_f} + \eta_1(t_f)}\\
	\beta_0 &= \frac{\rho_0}{\frac{1}{h} - \frac{b_1t_f}{h^2} + t_f\eta_2(t_f)}\\
	\gamma_0 &= \frac{\rho_0}{\frac{1}{h} + \frac{b_1t_f}{h^2} + t_f\eta_3(t_f)}.
\end{align*}

We now state the main result of our paper.

\begin{thm}\label{maintheorem}
	There exists an open neighbourhood $U\subset\mathbb{R}^2$ containing $(h, b_1) = (1.5, 0.1)$ such that for each $(\tilde h, \tilde{b}_1)\in U$, there is a unique solution $(a, b, c): [0,\infty)\rightarrow [0,\infty)\times (0, \infty)\times (0, \infty)$ of the $SU(2)$-invariant Einstein ODEs (\ref{tangential}), (\ref{conservation}) satisfying the boundary conditions $a(0)=0$, $\dot a(0)=2$, $b(0)=c(0)=\tilde h$ and $\dot b(0) = -\dot c(0) = \tilde{b}_1$. Thus, each $(\tilde h, \tilde{b}_1)\in U$ corresponds to a distinct	$SU(2)$-invariant complete negative Einstein metric $g_{(\tilde h, \tilde{b}_1)}$ on the bundle $\mathcal{O}(-4)$ over $\mathbb{C}P^1$ of the form
	\begin{equation*}
		g_{(\tilde h, \tilde{b}_1)}=dt^2 + a(t)^2 \sigma_1^2 + b(t)^2\sigma_2^2 + c(t)^2\sigma_3^2,
	\end{equation*}
	for $t\in [0, \infty)$. Each of these metrics is conformally compact and neither Kähler nor anti-self-dual.
\end{thm}
\begin{proof}
	Recall that so far we have used Lemma \ref{existencelemma} to construct an Einstein metric on the region $\{0\leq t< t_f\}$ around the singular orbit, with very good control on its value $\eta(t_f)$ at the ``endpoint'' $\{t=t_f\}$. This yields, after a coordinate and variables change, good control on $Z(s_0)$ via (\ref{Zs0expression}), as well as an estimate on $K_0$.
	
	By trial and error, following the heuristics outlined in the remark after the statement of Lemma \ref{infinitylemma}, we can choose $A, B, C, D$ such that the hypotheses of said lemma hold. Consequently, we obtain a solution $\eta(t)$ of the Einstein equation on $t\in [0,\infty)$ with the boundary parameters $(h, b_1) = (1.5, 0.1)$. 
	
	While we can deduce from the fact that it does not appear in the known classifications that $g_{(h, b_1)}$ is not Kähler nor anti-self-dual, it is also possible to verify this directly as follows. For each $t\in (0, t_f)$, we have the tight bound (\ref{mutfbounds}) on $|\mu(t)|$, and consequently good control on the functions $a(t), b(t), c(t)$ corresponding to $g_{(h, b_1)}$. One can explicitly check, using interval arithmetic, that the equations of Theorem 4.1 of \cite{DS94} are not satisfied, and consequently the metric is not Kähler. Finally, one checks that the metric is not anti-self-dual in a similar manner with equation (30) of \cite[\S 3]{CGLP03}. By the well-posedness of the initial value problem (\ref{etaODEslope2}) and the continuous nature of the estimates of \S\ref{infinitysection}, there must consequently be an open neighbourhood $U$ of $(h,b_1)$ in bijection with (non-Kähler, non-anti-self-dual) Einstein metrics.
	
	Finally, it remains to show that different points of $U$ correspond to distinct, non-isometric metrics. First, we note that diffeomorphisms of $M$ which are equivariant under the cohomogeneity one action (say of the Lie group $G$) must send singular orbits to singular orbits and principal orbits to principal orbits. Hence, if two metrics corresponding to two points in $U$ were isometric via a $G$-equivariant diffeomorphism, they must have isometric bolts, and consequently the same bolt size $\tilde h$. In order for principal orbits to be isometric in a neighbourhood of the bolt, they must also have the same value of $\tilde b_1$. 
	
	If two metrics $g_1$ and $g_2$ were isometric via a non $G$-equivariant diffeomorphism, then the isometric $G$-action on (say) $g_1$ pulls back to a $G$-action via isometries on $g_2$ which is different from the cohomogeneity one action. This implies that the isometry group of $g_2$ is strictly larger than $G$ (i.e. of higher dimension than $SU(2)$), which is not possible.

\end{proof}

\textit{Remark.} The metrics constructed in Theorem \ref{maintheorem} are genuinely triaxial in the sense that the initial data is a nonzero distance away (indeed a distance of $\approx0.1$) from the subset $\{b_1=0\}$ of parameter space corresponding to biaxial metrics. This is not the case for the metrics constructed in \cite[Lemma 4.5]{D25} by perturbing $U(2)$-invariant solutions.

\bigskip

\textbf{Acknowledgements.} I would like to thank Andrew Dancer and Jason Lotay for valuable discussions, guidance and comments on the manuscript. I am also grateful to Wolfgang Ziller for the fruitful discussions from which the initial ideas for this project emerged. Finally, I would like to thank the referees for their detailed feedback and helpful comments.

\appendix
\section{Expressions and estimates for $M_l$ and $M_{nl}$}\label{MlMnlsection}
Recall that $M(t,\eta) = (M_1, \dots, M_6)^T$ is an analytic function of $t$ and $\eta$ (for $\eta$ near $\hat \eta$) given by (\ref{Mformulas}), and that for a true Einstein metric $\eta$, an approximate Einstein metric $\hat\eta$, and $\mu = \eta - \hat\eta$, the quantities $M_l$ and $M_{nl}$ are defined by
\begin{equation*}
	M_l(t, \hat \eta) = \left(\frac{\partial M_i}{\partial \eta_j}(t,\hat\eta)\right)_{ij}, \qquad M_{nl}(t, \hat \eta, \mu) = M(t,\eta) - M(t,\hat\eta) - M_l(t, \hat\eta)\cdot \mu.
\end{equation*}
They form the decomposition 
\begin{equation*}
	M(t, \hat\eta + \mu) - M(t, \hat\eta) = \underbrace{M_l(t, \hat \eta)\cdot \mu}_{\mathcal{O}(\mu)} + \underbrace{M_{nl}(t, \hat\eta, \mu)}_{\mathcal{O}(\mu^2)},
\end{equation*}
which will play an essential role in the analysis of the Einstein equation near $\hat\eta$ (\ref{mutrueequation}). We are interested in obtaining reasonably sharp constants $C_l, C_{nl}>0$ satisfying the estimates (\ref{Cldfn}), (\ref{Cnldfn}). To do so, we compute in this appendix explicit expressions for $M_l$ and $M_{nl}$ from (\ref{Mformulas}).
\subsection{Linear term $M_l$}\label{jacobiansubsection}
We compute for $M_1, M_2, M_3$ that
\begin{gather*}
	\frac{\partial M_1}{\partial \eta_1} = -\hat{\eta}_4, \quad \frac{\partial M_1}{\partial \eta_4} = -\hat{\eta}_1\\
	\frac{\partial M_2}{\partial \eta_2} = -\left(\frac{b_1}{h} + \hat\eta_5\right), \quad \frac{\partial M_2}{\partial \eta_5} = \frac{b_1}{h^2} = \hat\eta_2\\
	\frac{\partial M_3}{\partial \eta_3} = \frac{b_1}{h} - \hat\eta_6, \quad \frac{\partial M_3}{\partial \eta_6} = -\left(\frac{b_1}{h^2} + \hat\eta_3\right)
\end{gather*}
with all other terms zero. Denote with a hat quantities corresponding to the approximate solution $\hat\eta$. For $M_4$, we have
\begin{gather*}
	\frac{\partial M_4}{\partial \eta_1} = -\frac{\hat X_2 \hat X_3}{\hat X_1^3} - \frac{\hat X_3^2\hat X_1}{\hat X_2^2} - \frac{\hat X_1 \hat X_2^2}{\hat X_3^2} + 2\hat X_1\\
	\frac{\partial M_4}{\partial \eta_2} = \left(\frac{\hat X_2\hat X_3^2}{\hat X_1^2} + \frac{\hat X_3^2\hat X_1^2}{\hat X_2^3} - \frac{\hat X_1^2 \hat X_2}{\hat X_3^2}\right)t\\
	\frac{\partial M_4}{\partial \eta_3} = \left(\frac{\hat X_3\hat X_2^2}{\hat X_1^2} + \frac{\hat X_2^2\hat X_1^2}{\hat X_3^3} - \frac{\hat X_1^2 \hat X_3}{\hat X_2^2}\right)t\\
	\frac{\partial M_4}{\partial \eta_4} = -2\hat\eta_4 - \hat\eta_5 - \hat\eta_6, \quad \frac{\partial M_4}{\partial \eta_5} = -\hat\eta_4, \quad \frac{\partial M_4}{\partial \eta_6} = -\hat\eta_4.
\end{gather*}
For $M_5$, we have
\begin{gather*}
	\frac{\partial M_5}{\partial \eta_1} = \frac{\hat X_1 \hat X_3^2}{\hat X_2^2} + \frac{\hat X_3^2 \hat X_2^2}{\hat X_1^3} - \frac{\hat X_2^2 \hat X_1}{\hat X_3^3}\\
	\frac{\partial M_5}{\partial \eta_2} = \left(-\frac{\hat X_1^2 \hat X_3^2}{\hat X_2^3} - \frac{\hat X_3^2\hat X_2}{\hat X_1^2} - \frac{\hat X_2 \hat X_1^2}{\hat X_3^2}+ 2\hat X_2\right)t + \frac{h}{2t}\\
	\frac{\partial M_5}{\partial \eta_3} = \left(\frac{\hat X_3 \hat X_1^2}{\hat X_2^2} + \frac{\hat X_1^2\hat X_2^2}{\hat X_3^3} - \frac{\hat X_2^2 \hat X_3}{\hat X_1^2}\right)t - \frac{h}{2t}\\
	\frac{\partial M_5}{\partial \eta_4} = -\left(\frac{b_1}{h} + \hat \eta_5\right), \quad \frac{\partial M_5}{\partial \eta_5} = -\hat\eta_4 - 2\hat\eta_5 - \hat\eta_6 - \frac{b_1}{h}, \quad \frac{\partial M_5}{\partial \eta_6} = -\left(\frac{b_1}{h} + \hat\eta_5\right).
\end{gather*}
For $M_6$, we have
\begin{gather*}
	\frac{\partial M_6}{\partial \eta_1} = \frac{\hat X_1\hat X_2^2}{\hat X_3^2} + \frac{\hat X_2^2 \hat X_3^2}{\hat X_1^3} - \frac{\hat X_3^2 \hat X_1}{\hat X_2^2}\\
	\frac{\partial M_6}{\partial \eta_2} = \left(\frac{\hat X_2\hat X_1^2}{\hat X_3^2} + \frac{\hat X_1^2\hat X_3^2}{\hat X_2^3} - \frac{\hat X_3^2\hat X_2}{\hat X_1^2}\right)t - \frac{h}{2t}\\
	\frac{\partial M_6}{\partial \eta_3} = \left(-\frac{\hat X_2^2\hat X_1^2}{\hat X_3^3} - \frac{\hat X_1^2\hat X_3}{\hat X_2^2} - \frac{\hat X_3 \hat X_2^2}{\hat X_1^2} + 2\hat X_3\right)t + \frac{h}{2t}\\
	\frac{\partial M_6}{\partial \eta_4} = \frac{b_1}{h} - \hat\eta_6,\quad	\frac{\partial M_6}{\partial \eta_5} = \frac{b_1}{h} - \hat\eta_6,\quad\frac{\partial M_6}{\partial \eta_6} = -\hat\eta_4 - \hat\eta_5 - 2\hat\eta_6 + \frac{b_1}{h}.
\end{gather*}
We may write this as a block matrix
\begin{equation*}
	M_l=\begin{pmatrix*}M_{00} & M_{01}\\
		M_{10} & M_{11}\end{pmatrix*},
\end{equation*}
with 
\begin{gather*}
	M_{00}=\diag\left\{-\hat\eta_4, -\left(\frac{b_1}{h} + \hat\eta_5\right), \frac{b_1}{h} - \hat\eta_6\right\}\\
	M_{01}=\diag\left\{-\hat\eta_1, \frac{b_1}{h^2} - \hat\eta_2, -\left(\frac{b_1}{h^2} + \hat\eta_3\right)\right\}\\
	M_{11}=\begin{pmatrix*}-2\hat\eta_4 - \hat\eta_5 - \hat\eta_6 & -\hat\eta_4 & -\hat\eta_4 \\
		-\left(\frac{b_1}{h} + \hat\eta_5\right) & -\frac{b_1}{h} - \hat\eta_4 - 2\hat\eta_5 - \hat\eta_6 & -\left(\frac{b_1}{h} + \hat\eta_5\right)\\
		\frac{b_1}{h} - \hat\eta_6 & \frac{b_1}{h} - \hat\eta_6 & \frac{b_1}{h} - \hat \eta_4 - \hat \eta_5 - 2\hat\eta_6\end{pmatrix*}.
\end{gather*}
We can loosely estimate
\begin{equation*}
	\norm{M_l}\leq \norm{M_{00}} + \norm{M_{01}} + \norm{M_{10}} + \norm{M_{11}}.
\end{equation*}
The first two terms are simple to estimate, namely
\begin{align*}
	\norm{M_{00}} &= \max\left\{\abs{\hat\eta_4}, \abs{\frac{b_1}{h} + \hat\eta_5}, \abs{\frac{b_1}{h} - \hat\eta_6} \right\}\\
	\norm{M_{01}} &= \max\left\{\abs{\hat\eta_1}, \abs{\frac{b_1}{h^2} - \hat\eta_2}, \abs{\frac{b_1}{h^2} + \hat\eta_3} \right\}.
\end{align*}
Furthermore
\begin{equation*}
	\norm{M_{11}} = -(\hat\eta_4+\hat\eta_5+\hat\eta_6)\mathrm{id} + \begin{pmatrix*}-\hat\eta_4 & -\hat\eta_4 & -\hat\eta_4\\
		-\frac{b_1}{h} - \hat\eta_5 & -\frac{b_1}{h} - \hat\eta_5 & -\frac{b_1}{h} - \hat\eta_5\\
		\frac{b_1}{h} - \hat\eta_6 & \frac{b_1}{h} - \hat\eta_6 & \frac{b_1}{h} - \hat\eta_6 \end{pmatrix*},
\end{equation*}
so that
\begin{equation*}
	\norm{M_{11}} \leq \abs{\hat\eta_4 +\hat\eta_5 + \hat\eta_6} + \sqrt{3}\sqrt{\abs{\hat\eta_4}^2 + \left(\frac{b_1}{h} + \hat\eta_5\right)^2 + \left(\frac{b_1}{h} - \hat\eta_6\right)^2}.
\end{equation*}
We further simplify the components of $M_{10}$, as terms with positive powers of $\hat X_1$ are \textit{a priori} singular. Explicitly smooth expressions are listed below.
\begin{align*}
	\frac{\partial M_4}{\partial \eta_1} &= -\hat X_2 \hat X_3 \frac{t^3}{(\frac{1}{2} + t\hat\eta_1)^3} - t\left(\frac{1}{2} + t\hat\eta_1\right)\frac{(\frac{2b_1}{h^2} + \hat\eta_3 - \hat\eta_2)(\hat X_2 + \hat X_3)^2}{\hat X_2^2 \hat X_3^2}\\
	\frac{\partial M_4}{\partial \eta_2} &= \hat X_2\hat X_3^2 \frac{t^3}{(\frac{1}{2} + t\hat\eta_1)^2}+ \frac{(2\frac{b_1}{h^2} + \hat\eta_3 - \hat\eta_2)(\hat X_3^2 + \hat X_2^2)(\hat X_3+\hat X_2)}{\hat X_2^3\hat X_3^2}\left(\frac{1}{2} + t\hat\eta_1\right)^2\\
	\frac{\partial M_4}{\partial \eta_3} &= \hat X_2^2\hat X_3 \frac{t^3}{(\frac{1}{2} + t\hat\eta_1)^2}+ \frac{-(2\frac{b_1}{h^2} + \hat\eta_3 - \hat\eta_2)(\hat X_3^2 + \hat X_2^2)(\hat X_3+\hat X_2)}{\hat X_2^2\hat X_3^3}\left(\frac{1}{2} + t\hat\eta_1\right)^2
\end{align*}
\begin{align*}
	\frac{\partial M_5}{\partial \eta_1} &= \hat X_3^2\hat X_2^2 \frac{t^3}{(\frac{1}{2}+t\hat\eta_1)^3}+ \left(\frac{1}{2} + t\hat\eta_1\right)\frac{(\hat X_3^2 + \hat X_2^2)(\hat X_3 + \hat X_2)}{\hat X_2^2 \hat X_3^2}\left(\frac{2b_1}{h^2} + \hat\eta_3 - \hat\eta_2 \right)\\
	\frac{\partial M_5}{\partial \eta_2} &= 2\hat X_2 t - \hat X_3^2 \hat X_2 \frac{t^3}{(\frac{1}{2} + t\hat\eta_1)^2} - \hat\eta_1 (1+t\hat\eta_1)\left(\frac{\hat  X_3^2}{\hat X_2^3} + \frac{\hat X_2}{\hat X_3^2}\right)\\
	&\quad +\frac{1}{4\hat X_2^3} \Bigg[\left(-\frac{5b_1}{h^3} + \frac{1}{h}(3\hat\eta_2 - 2\hat\eta_3)\right) + t\left(3\left(\frac{-b_1}{h^2} + \hat\eta_2\right)^2-\left(\frac{b_1}{h^2} + \hat\eta_3\right)^2\right) \\
	&\quad + t^2h\left(-\frac{b_1}{h^2} + \hat\eta_2\right)^3\Bigg] + \frac{1}{4\hat X_3^2} \left[\frac{3b_1}{h^2} - \hat\eta_2 + \hat\eta_3 + th\left(\frac{b_1}{h^2} + \hat\eta_3\right)^2\right]\\
	\frac{\partial M_5}{\partial \eta_3} &= - \hat X_2^2 \hat X_3 \frac{t^3}{(\frac{1}{2} + t\hat\eta_1)^2} + \hat\eta_1 (1+t\hat\eta_1)\left(\frac{\hat  X_2^2}{\hat X_3^3} + \frac{\hat X_3}{\hat X_2^2}\right)\\
	&\quad -\frac{1}{4\hat X_3^3} \Bigg[\left(\frac{5b_1}{h^3} + \frac{1}{h}(3\hat\eta_3 - 2\hat\eta_2)\right) + t\left(3\left(\frac{b_1}{h^2} + \hat\eta_3\right)^2-\left(-\frac{b_1}{h^2} + \hat\eta_2\right)^2\right) \\
	&\quad + t^2h\left(\frac{b_1}{h^2} + \hat\eta_3\right)^3\Bigg] - \frac{1}{4\hat X_2^2} \left[-\frac{3b_1}{h^2} - \hat\eta_3 + \hat\eta_2 + th\left(-\frac{b_1}{h^2} + \hat\eta_2\right)^2\right]
\end{align*}
\begin{align*}
	\frac{\partial M_6}{\partial \eta_1} &= \hat X_3^2\hat X_2^2 \frac{t^3}{(\frac{1}{2}+t\hat\eta_1)^3}- \left(\frac{1}{2} + t\hat\eta_1\right)\frac{(\hat X_3^2 + \hat X_2^2)(\hat X_3 + \hat X_2)}{\hat X_2^2 \hat X_3^2}\left(\frac{2b_1}{h^2} + \hat\eta_3 - \hat\eta_2 \right)\\
	\frac{\partial M_6}{\partial \eta_2} &=	- \hat X_3^2 \hat X_2 \frac{t^3}{(\frac{1}{2} + t\hat\eta_1)^2} + \hat\eta_1 (1+t\hat\eta_1)\left(\frac{\hat  X_3^2}{\hat X_2^3} + \frac{\hat X_2}{\hat X_3^2}\right)\\
	&\quad -\frac{1}{4\hat X_2^3} \Bigg[\left(-\frac{5b_1}{h^3} + \frac{1}{h}(3\hat\eta_2 - 2\hat\eta_3)\right) + t\left(3\left(\frac{-b_1}{h^2} + \hat\eta_2\right)^2-\left(\frac{b_1}{h^2} + \hat\eta_3\right)^2\right) \\
	&\quad + t^2h\left(-\frac{b_1}{h^2} + \hat\eta_2\right)^3\Bigg] - \frac{1}{4\hat X_3^2} \left[\frac{3b_1}{h^2} - \hat\eta_2 + \hat\eta_3 + th\left(\frac{b_1}{h^2} + \hat\eta_3\right)^2\right]\\
	\frac{\partial M_6}{\partial \eta_3} &= 2\hat X_3 t - \hat X_2^2 \hat X_3 \frac{t^3}{(\frac{1}{2} + t\hat\eta_1)^2} - \hat\eta_1 (1+t\hat\eta_1)\left(\frac{\hat  X_2^2}{\hat X_3^3} + \frac{\hat X_3}{\hat X_2^2}\right)\\
	&\quad +\frac{1}{4\hat X_3^3} \Bigg[\left(\frac{5b_1}{h^3} + \frac{1}{h}(3\hat\eta_3 - 2\hat\eta_2)\right) + t\left(3\left(\frac{b_1}{h^2} + \hat\eta_3\right)^2-\left(-\frac{b_1}{h^2} + \hat\eta_2\right)^2\right) \\
	&\quad + t^2h\left(\frac{b_1}{h^2} + \hat\eta_3\right)^3\Bigg] + \frac{1}{4\hat X_2^2} \left[-\frac{3b_1}{h^2} - \hat\eta_3 + \hat\eta_2 + th\left(-\frac{b_1}{h^2} + \hat\eta_2\right)^2\right].
\end{align*}
For simplicity, we estimate the operator norm of $M_{10}$ by its Frobenius norm, namely
\begin{equation*}
	\norm{M_{10}}^2\leq \sum_{i=4}^6\sum_{j=1}^3 \left(\frac{\partial M_i}{\partial\eta_j}\right)^2.
\end{equation*}

\subsection{Nonlinear term $M_{nl}$}
We have
\begin{equation*}
	M_{nl,1}=-\mu_1\mu_4, \quad M_{nl,2}=-\mu_2\mu_5, \quad M_{nl, 3}= -\mu_3\mu_6,
\end{equation*}
as well as
\begin{align*}
	M_{nl,4}&=-\mu_4(\mu_4+\mu_5+\mu_6) + R_{1,nl}\\
	M_{nl,5}&=-\mu_5(\mu_4+\mu_5+\mu_6) + R_{2,nl}\\
	M_{nl,6}&=-\mu_6(\mu_4+\mu_5+\mu_6) + R_{3,nl},
\end{align*}
where $R_{i,nl}$ is the nonlinear part of $R_i(X)-R_i(\hat X)$ considered as a function of $\mu$.

Let $\kappa^{ij}_k$ be the nonlinear part of $\frac{X_i^2 X_j^2}{X_k^2} - \frac{\hat X_i^2 \hat X_j^2}{\hat X_k^2}$. Then
\begin{align*}
	R_{1,nl}&=\frac{1}{2}(\kappa^{23}_1 - \kappa^{31}_2 - \kappa^{21}_3) + \mu_1^2\\
	R_{2,nl}&=\frac{1}{2}(\kappa^{31}_2 - \kappa^{21}_3 - \kappa^{23}_1) + t^2\mu_2^2\\
	R_{3,nl}&=\frac{1}{2}(\kappa^{21}_3 - \kappa^{23}_1 - \kappa^{31}_2) + t^2\mu_3^2.
\end{align*}
An explicit computation yields that
\begin{align*}
	\kappa^{23}_1 &= \left(\frac{1}{\hat X_1^2} - \frac{\mu_1(2\hat X_1 + \mu_1)}{\hat X_1^2 (\hat X_1 +\mu_1)^2}\right)\left(t^2(\hat X_2^2\mu_3^2 + \hat X_3^2\mu_2^2) + (2\hat X_2 t\mu_2 + t^2\mu_2^2)(2\hat X_3t\mu_3 + t^2\mu_2^2)\right) \\
	&\quad - \frac{\mu_1(2\hat X_1+\mu_1)}{\hat X_1^2 (\hat X_1 + \mu_1)^2}\cdot 2t(\hat X_2^2 \hat X_3\mu_3 + \hat X_3^2 \hat X_2 \mu_2)\\
	\kappa^{31}_2 &= \left(\frac{1}{2}+t\hat\eta_1\right)^2\left(\frac{\mu_2^2\hat X_3^2(3\hat X_2 + 2t\mu_2)}{\hat X_2^3(\hat X_2 + t\mu_2)^2} - \frac{4\hat X_3 \mu_2\mu_3}{\hat X_2^2} + \frac{2\hat X_3 t\mu_3\mu_2^2(3\hat X_2 + 2t\mu_2)}{\hat X_2^3 (\hat X_2 + t\mu_2)^2}\right)\\
	&\quad + \frac{2\mu_1(\frac{1}{2} + t\hat\eta_1) + t\mu_1^2}{\hat X_2^2(\hat X_2 + t\mu_2)^2} \left((2\hat X_3 \mu_3 + t\mu_3^2)\hat X_2^2 - (2\hat X_2 \mu_2 + t\mu_2^2)\hat X_3^2\right) + \mu_1^2\frac{\hat X_3^2}{\hat X_2^2}.
\end{align*}
The other relevant coefficient $\kappa^{21}_3$ is obtained by a permutation of the indices $2\leftrightarrow 3$ in the formula for $\kappa^{31}_2$.

\section{Specific values used in the computer-assisted proof}\label{valuesapp}

\begin{table}[hbt!]
	\begin{tabular}{|c|c|c|}
		\hline
		Constant & Value & Requirement\\ \hline
		$h$ & $\frac{3}{2} = 1.5$ & $>\frac{1}{2}\sqrt{\frac{31}{23}}$ and $<2+\sqrt{3}$\\
		$b_1$ & $\frac{1}{10} = 0.1$ & \\
		$t_f$ & 2.25 & \\
		$C_l$ & $< 30.895$ & \\
		$C_{nl}$ & $< 12.620$ & \\
		$\lambda$ & $\frac{1}{12}$ & \\
		$t_0$ & $>9.100 \cdot 10^{-17}$ & \\
		$\mathcal{I}(t_f)$ & $<2.783 \cdot 10^{28}$& \\
		$B$ &$\frac{19}{8} + \frac{1}{e} + \frac{1}{2}\left(h + \frac{1}{h}\right) \approx 3.826$ & \\
		$\varepsilon_0$ &  $>6.461 \cdot 10^{-32}$ & \\ 
		$\varepsilon$ & $<2.709\cdot 10^{-33}$ & $<\varepsilon_0$ \\ \hline

	\end{tabular}
	\caption{Constants and values used in the fixed-point-theorem estimates of \S\ref{fixedpointsection}. The hypotheses mentioned are those of Lemma \ref{existencelemma}. Quantities listed without inequalities are exact.}
\end{table}

\begin{table}[hbt!]
	\begin{tabular}{|c|c|c|}
		\hline
		Constant & Value & Requirement\\ \hline
		$A$ & $0.375$ & \\
		$B$ & $0.43$ & \\
		$C$ & $-2.5$ & \\
		$D$ & $K_0+\delta$ &  for any $\delta>0$ \\
		$2C+\frac{1}{B^2}$ & $\frac{755}{1849}$ & $\geq 0$ from hypothesis (\ref{inf1})\\
		$s_0$ & $1-\tanh(1.125) \approx 0.191$ & \\
		$\abs{Z(s_0)}$ & $<0.199$ & \\
		$\zeta_m$ & $<0.332\, 94$ & $<1/3$ from hypothesis (\ref{inf2})\\
		$4 + C - s_0/(2-s_0) - \sqrt{3}\zeta_m$ & $>0.817$ &  $>0$ from hypothesis (\ref{inf3})\\
		$K_0$ & $<0.594$ & $<D$ from hypothesis (\ref{inf4})\\ \hline
	\end{tabular}
	\caption{Constants and values used in the estimates at infinity of \S\ref{infinitysection}. The hypotheses mentioned are those of Lemma \ref{infinitylemma}. Quantities listed without inequalities are exact.}
\end{table}

In this appendix, we detail for illustrative purposes the specific values for the various constants and estimates obtained from the choice of initial data $h=1.5$, $b_1=0.1$.

The heuristic ODE solver is used at 1000 digit precision, and the numerical Sturm's theorem estimator uses 2000 digit precision. Each power series is evaluated up to order 110, and then each component of the heuristic solution $\tilde\eta$ is approximated by a degree 110 Chebyshev polynomial.

The specific solution we produce at $(h, b_1) = (1.5, 0.1)$ is shown in Figure \ref{metricfig}.

\begin{figure}[hbt!]
	\centering
	\begin{subfigure}[t]{0.48\textwidth}
		\centering
		\includegraphics[width=\linewidth]{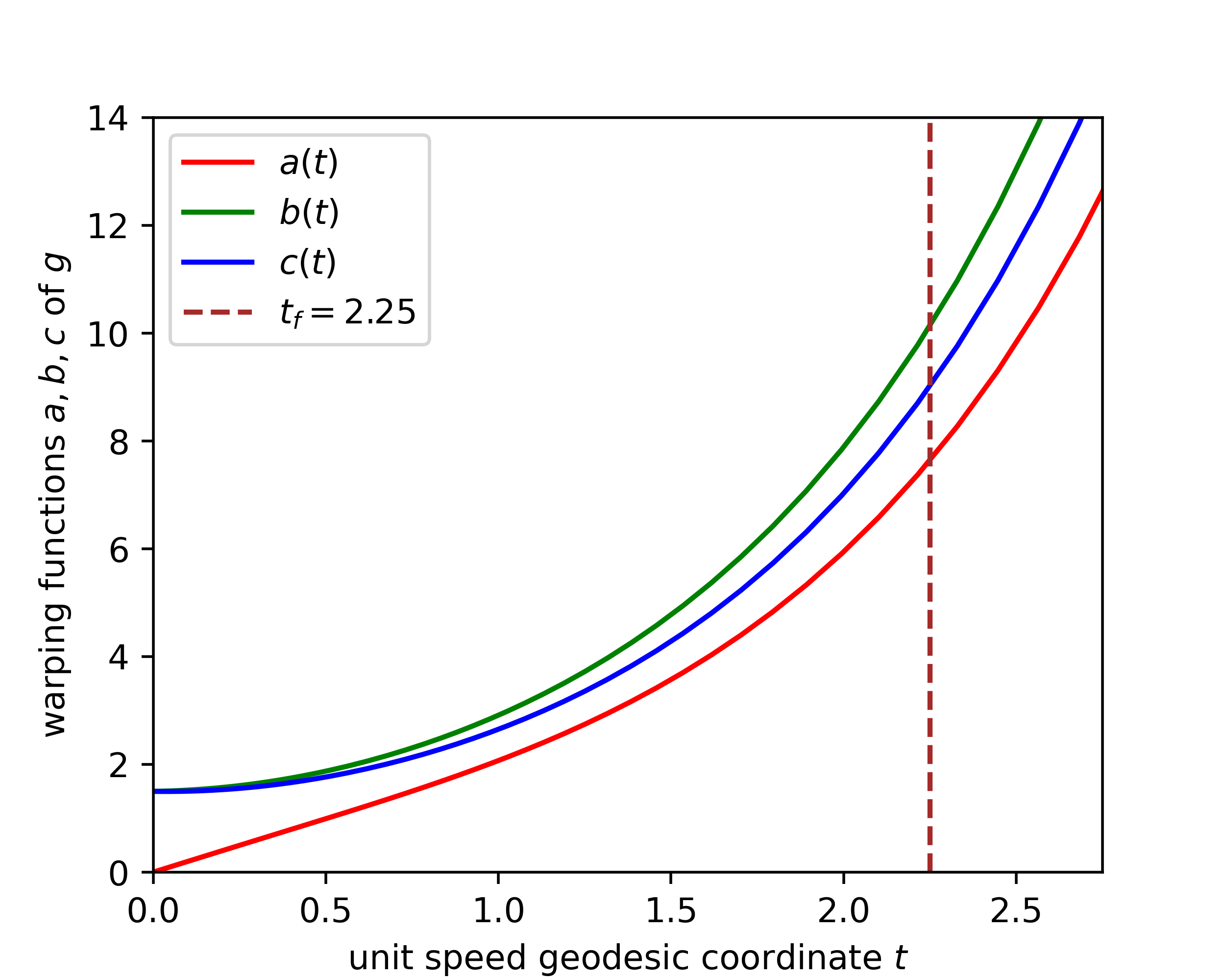} 
		\caption{Warping functions of an Einstein metric $g$ of the form (\ref{b9}) in the geodesic coordinate $t$. Note that $a, b, c \sim e^t$ as $t\to\infty$.}
	\end{subfigure}
	\hfill
	\begin{subfigure}[t]{0.48\textwidth}
		\centering
		\includegraphics[width=\linewidth]{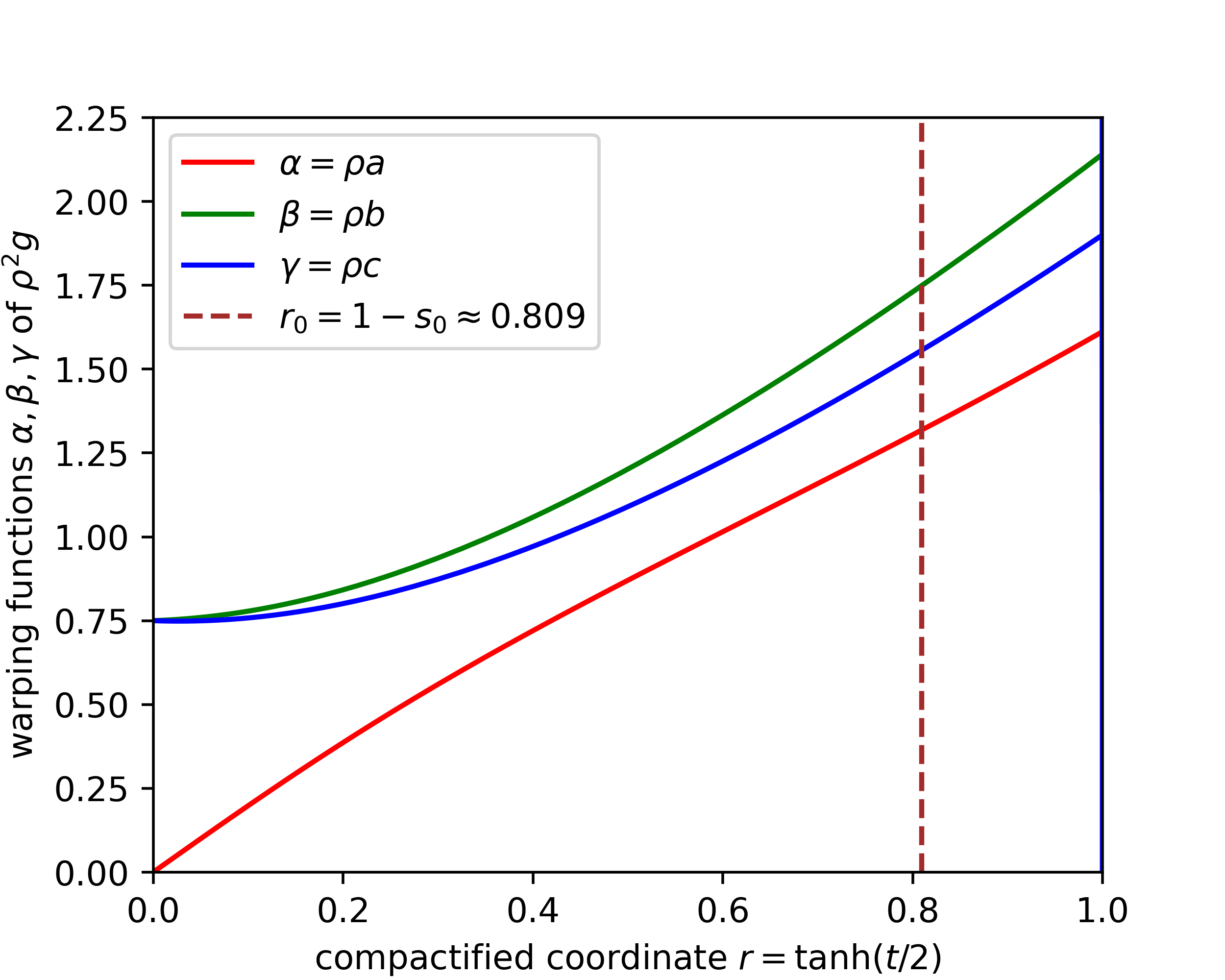} 
		\caption{Warping functions of $\rho^2g$ as a function of the compact coordinate $r\in [0,1]$. Note that $\alpha'(1) = \alpha(1)$ and cyclically.}
	\end{subfigure}
	\caption{Plot of the warping functions of the Einstein metric constructed in \S \ref{computersection}, with $h=1.5$, $b_1=0.1$. The cutoff time where the fixed-point construction of \S\ref{fixedpointsection} meets the asymptotic hyperbolicity estimates of \S\ref{infinitysection} is indicated by the brown dashed line. The numerics are produced using an order 8 Runge--Kutta solver.}\label{metricfig}
\end{figure}

\section{Boundary data and formal power series solutions for each singular orbit topology}\label{boundaryappendix}

In this appendix, we will discuss the initial data to be used by the heuristic numerical solver for various types of singular orbit topologies. From this point on, for simplicity we will denote $b_j=b^{(j)}(0)$ and so on.

\subsection{Initial data for a nut}

Suppose that we have a nut at $t=0$, i.e. $a(0)=b(0)=c(0)=0$, $\dot a(0)=\dot b(0)=\dot c(0) = \frac{1}{2}$ and the singular orbit is a point. By the smoothness conditions, $\ddot a(0)=\ddot b(0)=\ddot c(0)=0$. In fact, even without assuming that the functions are odd, one can show that the equations require the second derivatives at zero to vanish. In order to distinguish $a,b$ and $c$, we thus have to consider the order 3 term. The tangential ODEs are trivial, while the conservation law gives
\begin{equation*}
	a_3+b_3+c_3 = -\frac{\Lambda}{2}.
\end{equation*}
Consequently, the metric near $t=0$ takes the form
\begin{gather*}
	a(t)= \frac{1}{2}t + \frac{a_3}{6}t^3 + \mathcal{O}(t^5), \qquad b(t)= \frac{1}{2}t + \frac{b_3}{6}t^3 + \mathcal{O}(t^5),\\
	c(t)= \frac{1}{2}t - \frac{1}{6}\left(\frac{\Lambda}{2} + a_3 + b_3\right)t^3 + \mathcal{O}(t^5).
\end{gather*}
The free parameters are $a_3$ and $b_3$.

\bigskip

Consider now a bolt at $t=0$, supposing without loss of generality that $a(0)=0$ while $b(0),c(0)\neq 0$. Owing to the symmetries of the equations, we may assume without loss of generality that $a,b,c>0$. All of the constraints and equations also are symmetric under the exchange of $b$ and $c$. Many of the calculations in this section are done with a computer algebra system. 

First, we note that in the case of a bolt, the conservation equation (\ref{conservation}) is entirely redundant, as it only needs to be verified to zeroth order and at zeroth order is equivalent to the sum of the $\ddot b$ and $\ddot c$ equations.

Taking limits $t\to 0$ on both sides of the equation (\ref{tangentialconformal}) for $\ddot a$, we get that $b(0)=c(0)$ and that $\ddot a = 0$. Denote 
\begin{equation*}
	h\equiv b(0)=c(0).
\end{equation*}

In fact, the smoothness conditions require that $a$ be odd and $b^2,c^2$ be even, although it suffices to verify this up to second order. For $b,c$, we note that the equation for $\ddot a =0 $ gives that $b_1+c_1=0$, while the right-hand side of the equation for $\ddot b$ (or indeed for $\ddot c$) tends as $t\to 0$ to $b_1\left(\frac{1}{a_1^2}-4\right)$. Therefore $b_1=c_1=0$, unless $a_1=2$.

\subsection{Initial data for an $\dot a(0)=2$ bolt (total space $\mathcal{O}(-4)$)}\label{slope2data}

Assume that $a_1=2$. In this case, we have the free parameters $h$ and $b_1$, and the formal power series solution
\begin{gather*}
	a(t)= 2t + \mathcal{O}(t^3), \quad b(t) =  h + b_1t + \mathcal{O}(t^2), \quad c(t)= h-b_1t + \mathcal{O}(t^2).
\end{gather*}
The ranges of the parameters are $h\in(0,\infty)$ and $b_1\in[0,\infty)$.

\bigskip

Assume henceforth that $a_1\neq 2$. This results in $b_1=c_1=0$ and subsequently the $\ddot b=0$ equation gives that
\begin{equation*}
	\left(2-\frac{1}{\dot a(0)^2}\right)\ddot b(0) + \frac{1}{\dot a(0)^2}\ddot c(0)=\frac{1}{h}-\Lambda h
\end{equation*}
and the $\ddot c = 0$ equation gives the same thing with $b$ and $c$ swapped. The two above conditions together imply that for all $a_1\neq 1$,
\begin{equation*}
	b_2=c_2=\frac{1}{2}\left(\frac{1}{h}-\Lambda h\right).
\end{equation*}

\subsection{Initial data for a $\dot a(0)=1$ bolt (total space $\mathcal{O}(-2)$)}

When $a_1=1$, we just have the constraint $b_2+c_2=\frac{1}{h} -\Lambda h$. Therefore the free parameters in question are $h$ and $b_2$, and the initial data is
\begin{gather*}
	a(t)= t + \mathcal{O}(t^3), \qquad b(t)= h + \frac{b_2}{2}t^2 + \mathcal{O}(t^3)\\
	c(t)= h + \left(\frac{1}{h} - \Lambda h - b_2\right) \frac{t^2}{2} + \mathcal{O}(t^3).
\end{gather*}

\bigskip

Assume from now on that $a_1\neq 2, 1$.

We now consider first order constraints, i.e. evaluating the first time derivative of the tangential equations at zero. The derivative of the $\ddot a$ equation yields $a_3=-a_1/h^2$. The first derivative of the $\ddot b$ (and therefore $\ddot c$) equation holds trivially for all values of the parameters.

Now we move to second order constraints, given $a_3=-a_1/h^2$. The second derivative of the $\ddot a=0$ equation is trivially true, and the second derivative of the $\ddot b=0$ and $\ddot c=0$ equations respectively give us
\begin{align*}
	\left(8-\frac{1}{a_1^2}\right)b_4 + \frac{1}{a_1^2}c_4 &= -\frac{6a_1^2+4}{h^3} + \frac{4\Lambda}{h}\\
	\frac{1}{a_1^2}b_4 + \left(8-\frac{1}{a_1^2}\right)c_4 &= -\frac{6a_1^2+4}{h^3} + \frac{4\Lambda}{h}.
\end{align*}
We see that $b_4=c_4=-\frac{3}{4}\frac{a_1^2}{h^3} - \frac{1}{2h^3} + \frac{\Lambda}{2h}$ if and only if $a_1\neq 1/2$. 

\subsection{Initial data for a $\dot a(0)=1/2$ bolt (total space $\mathcal{O}(-1)$)}
If $a_1=1/2$, then we only have the constraint
\begin{equation*}
	b_4+c_4=-\frac{3}{2}\frac{a_1^2}{h^3} - \frac{1}{h^3} + \frac{\Lambda}{h} = -\frac{11}{8h^3} + \frac{\Lambda}{h}.
\end{equation*}
Correspondingly, the solution takes the form
\begin{gather*}
	a(t)= \frac{t}{2} - \frac{t^3}{6h^2} + \mathcal{O}(t^5), \quad b(t)= h + \left(\frac{1}{h} - \Lambda h\right)\frac{t^2}{4} + \frac{b_4}{24}t^4 + \mathcal{O}(t^5),\\
	c(t)= h + \left(\frac{1}{h} - \Lambda h\right)\frac{t^2}{4} + \frac{1}{24}\left(-\frac{11}{8h^3}+\frac{\Lambda}{h}-b_4\right)t^4 + \mathcal{O}(t^5).
\end{gather*}

\bigskip

\textbf{Statements and declarations.} No data was collected or used in this work. The author declares no competing interests.

\bibliography{refs}
\bibliographystyle{amsplain}

\end{document}